\documentclass{amsart}
\usepackage[dvipsnames]{xcolor} 
\usepackage[urlcolor=pink, 
linkcolor=blue,
citecolor=Emerald,
colorlinks=true,
linktocpage=true]{hyperref}
\usepackage{tikz}
\usepackage{graphicx,amssymb,eucal,mathrsfs}
\usepackage{latexsym,amsmath,epsfig,epic,eepic}
\usepackage{amscd}
\usepackage{amsthm}
\usepackage{indentfirst}
\usepackage{epsf,pgf}
\usepackage{graphicx}
\usepackage{pstricks,mathrsfs}

\title[On certain $C^0$-aspects of contactomorphism groups]{On certain $C^0$-aspects of contactomorphism groups}

\author[Baptiste Serraille]{Baptiste Serraille}
\address{Department of Mathematics, ETH-Z\"{u}rich, R\"{a}mistrasse 101, 8092 Z\"{u}rich, Switzerland.}
\email{baptiste.serraille@math.ethz.ch}

\author[Vuka\v sin Stojisavljevi\'c ]{Vuka\v sin Stojisavljevi\'c}
\address{D\'e\-par\-te\-ment de math\'ematiques et de
sta\-tistique, Univer\-sit\'e de Mont\-r\'eal,  CP 6128 succ
Centre-Ville, Mont\-r\'eal,  QC  H3C 3J7, Canada.}
\email{vukasin.stojisavljevic@gmail.com}

\theoremstyle{theorem}
\newtheorem{thm}{Theorem}
\newtheorem{lemma}[thm]{Lemma}
\newtheorem{coro}[thm]{Corollary}
\newtheorem{prop}[thm]{Proposition}
\newtheorem{conjecture}[thm]{Conjecture}
\newtheorem{question}[thm]{Question}
\newtheorem{ex}[thm]{Example}
\newtheorem{dfn}[thm]{Definition}
\newtheorem{rem}[thm]{Remark}

\numberwithin{thm}{section} 

\newcommand{\Ham}{\text{Ham}}
\newcommand{\HamH}{\overline{\text{Ham}}_{c}}
\newcommand{\Homeo}{\text{Homeo}}
\newcommand{\Contoc}{\text{Cont}_{0,c}}
\newcommand{\ContocH}{\overline{\text{Cont}}_{0,c}}
\newcommand{\ContocU}{\widetilde{\text{Cont}}_{0,c}}

\newcommand{\conj}{\text{Conj}}
\newcommand{\conjH}{\overline{\text{Conj}}}

\newcommand{\diam}{\text{diam}}

\newcommand{\supp}{\text{supp}}

\newcommand{\R}{\mathbb{R}}

\newcommand{\Z}{\mathbb{Z}}

\begin{document}

\begin{abstract}
We study a number of questions related to the $C^0$-topology of contactomorphisms and contact homeomorphisms. In particular, we show a connection between Rokhlin property of contact homeomorphisms and contact non-squeezing, we define a new conjugation-invariant norm on contactomorphisms and explore its relation to the contact fragmentation norm and we introduce a measure of the size of conjugacy classes which is related to weak conjugacy equivalence. We also show that Sandon's spectral norm is $C^0$-locally bounded and extend its definition to contact homeomorphisms.
\end{abstract}

\maketitle	
\tableofcontents

\section{Introduction}

The present paper concerns aspects of $C^0$-contact topology reflected in contactomorphism groups and their $C^0$-closures. The field of $C^0$-contact topology studies the behaviour of smooth objects, such as contactomorphisms or Legendrian submanifolds, under $C^0$-limits.  It is analogous to the much more studied domain of $C^0$-symplectic topology. The starting point of both theories is the celebrated Eliashberg-Gromov theorem and its contact analogue, see \cite{Gro86,Eli87}. It states that a diffeomorphism obtained as a $C^0$-limit of symplectomorphisms is itself a symplectomorphism and similarly for a $C^0$-limit of contactomorphisms, see \cite{MS14,Mul19} for a detailed treatment of the contact case.  Motivated by these results, one defines a \textit{contact homeomorphism} as a homeomorphisms which can be obtained as a $C^0$-limit of smooth contactomorphisms. Recently, significant progress has been made in understanding the action of contact homeomorphisms on Legendrian submanifolds, see \cite{RZ20,Ush21, Naka20,DRS24,DRS24B, Sto22,DRS22,Sto24}. Despite these advances, the understanding of contact homeomorphisms still remains in the early stages.  We focus on studying $C^0$-properties of conjugacy classes of  contactomorphisms and contact homeomorphisms from quantitative perspective.

\subsection{Rokhlin property}\label{Section:Intro_Rokhlin}

Given a cooriented\footnote{We assume all contact manifolds to be cooriented, unless specified otherwise.} contact manifold $(Y,\xi=\ker \alpha)$, we denote by $\Contoc(Y)$ the identity component of the group of compactly supported contactomorphisms of $(Y,\xi).$ This is a subgroup of the group of compactly supported homeomorphisms of $Y$, denoted by $\Homeo_c(Y).$ 

We fix a Riemannian metric on $Y$ and denote by $d$ the induced distance on $Y.$ The $C^0$-distance on $\Homeo_c(Y)$ is defined as
$$d_{C^0}(\phi, \psi)=\max_{x\in Y} d(\phi(x),\psi(x)). $$
The corresponding $C^0$-norm will be denoted by $\| \cdot \|_{C^0}.$ Lastly, we denote by $\tau_{C^0}$ the topology on $\Homeo_c(Y)$ induced by $d_{C^0}.$ This topology renders $\Homeo_c(Y)$ into a topological group, see Lemma \ref{Lemma:Top_Group}. Denote by $\ContocH(Y)$ the closure of $\Contoc(Y)$ inside $(\Homeo_c(Y), \tau_{C^0}).$ Since $(\Homeo_c(Y), \tau_{C^0})$ is a topological group, $(\ContocH(Y), \tau_{C^0})$ is a topological group as well.

\begin{rem}
If $Y$ is compact, $\tau_{C^0}$ coincides with the compact-open topology as well as with the strong $C^0$-topology, both of which are canonically defined, without referring to an auxiliary metric on $Y.$ In the non-compact case, which we are mainly focused on, all of these topologies are different and comparable. Moreover, $\tau_{C^0}$ itself depends on the choice of a metric on $Y.$ A discussion of different topologies on $\Homeo_c(Y)$ can be found in Subsection \ref{Section_Maps_Topologies}. 
\end{rem}

Let $(x_1,\ldots ,x_n,y_1, \ldots,y_n,z)$ be coordinates on $\R^{2n+1}.$ We equip  $\R^{2n+1}$ and $\R^{2n}\times S^1= \R^{2n} \times \R / \Z $ with the standard contact structure given by
$$\xi_0 = \ker \alpha_0, ~ \alpha_0 = dz - \sum_{i=1}^n y_i dx_i.$$
When referring to the $C^0$-norm, $C^0$-distance or $\tau_{C^0}$, we assume that $\R^{2n+1}$ is equipped with the standard Euclidean metric and $\R^{2n}\times S^1$ with the induced metric on the quotient $\R^{2n}\times S^1=\R^{2n}\times \R / \Z .$

\begin{dfn}
Let $G$ be a topological group. We say that $G$ has Rokhlin property if there exists $g\in G$ such that $\conj(g):= \{ hgh^{-1}~|~h \in G \}$ is dense in $G$.
\end{dfn}

The definition of Rokhlin property of topological groups was first introduced in \cite{GW01,GW08}. We refer the reader to these papers for context and examples.  In symplectic topology,  Rokhlin property of Hamiltonian homeomorphisms of surfaces was studied in \cite{Sey13,LRSV21}, following a question of B\'{e}guin, Crovisier and Le Roux, see also \cite{GG04,EPP12}.  Our first goal is to explore footprints of the Rokhlin property in $C^0$-contact topology.  The following theorem illustrates a connection between the Rokhlin property of $\ContocH$ and the dichotomy between contact squeezing and non-squeezing.

\begin{thm}\label{thm: Rokhlin property of the contactomorphisms of the ball}
$(\ContocH(\R^{2n+1}), \tau_{C^0})$ has Rokhlin property. On the other hand, $(\ContocH(\R^{2n}\times S^1), \tau_{C^0})$ does not have Rokhlin property.
\end{thm}

As we mentioned,  Theorem \ref{thm: Rokhlin property of the contactomorphisms of the ball} shows a link between group properties of $\ContocH$ and contact flexibility, given by contact squeezing, or contact rigidity, given by contact non-squeezing.  More concretely, the proof of the first part of this theorem relies on an explicit construction of a contact homeomorphism whose conjugacy class is $C^0$-dense in $\ContocH(\R^{2n+1}).$ The construction is possible due to the fact that balls in $\R^{2n+1}$ can be arbitrarily squeezed by elements of $\Contoc(\R^{2n+1}).$ This is not the case for $\R^{2n}\times S^1$, as shown by the celebrated contact non-squeezing theorem of Eliashberg, Kim and Polterovich \cite{EKP06}.  The proof of the second part of Theorem \ref{thm: Rokhlin property of the contactomorphisms of the ball} is an application of the contact non-squeezing theorem. 

\begin{rem}
Let $\Ham_c(\R^{2n})$ be the group of compactly supported Hamiltonian diffeomorphisms of $(\R^{2n},dx \wedge dy)$ and denote by $\HamH(\R^{2n})$ the $C^0$-closure of $\Ham_c(\R^{2n})$ inside $\Homeo_c(\R^{2n}).$ In contrast to the contact case, an elementary argument implies that $(\HamH(\R^{2n}), \tau_{C^0})$ does not have Rokhlin property.  Namely, if $\phi,\psi \in \HamH(\R^{2n})$ we have that $\supp (\psi \phi \psi^{-1})=\psi (\supp (\phi))$ and thus the volume of the interior of the support is an invariant under conjugation, which prevents Rokhlin property.  The same argument shows that the group of compactly supported volume preserving homeomorphisms of an arbitrary open manifold does not have Rokhlin property.
\end{rem}

\begin{rem} Let $B(r)\subset \R^{2n}$ denote the open ball of radius\footnote{This notation is sometimes used for a ball of capacity $r.$ Throughout the paper we also consider balls in $\R^{2n+1}$ and thus choose the notation which makes sense in both cases. } $r.$ The same argument we use to prove the second part of Theorem \ref{thm: Rokhlin property of the contactomorphisms of the ball} shows that $(\ContocH(B(r)\times S^1), \tau_{C^0})$ does not have Rokhlin property for any $r>\frac{1}{\sqrt{\pi}}.$ It is based on the fact that contact non-squeezing holds in $B(r)\times S^1$ for $r>\frac{1}{\sqrt{\pi}}$, but not for $r\leq \frac{1}{\sqrt{\pi}}.$ This motivates the following question:
\begin{question}
Does $(\ContocH(B(r)\times S^1), \tau_{C^0})$ have Rokhlin property for some $r\leq \frac{1}{\sqrt{\pi}}$?
\end{question}
\end{rem}

\subsection{Spectral norm on contact homeomorphisms}\label{Sec:C0Sandon}


In \cite{San2}, Sandon defined a conjugation-invariant norm $\gamma:\Contoc(\R^{2n}\times S^1) \to \Z.$ The definition is analogous to the one of the spectral norm on $\Ham_c(\R^{2n})$ given by Viterbo in \cite{Vit92}. Both definitions rely on spectral invariants extracted from the generating function homology and we refer to them as spectral norms on the respective groups.

Viterbo's definition has been extended to more general symplectic manifolds using spectral invariants coming from Floer theory, see \cite{Sch00,Oh05}. In recent years, a number of works established $C^0$-continuity of the spectral norm on $\Ham$ for various classes of symplectic manifolds, see \cite{Vit92,Sey13,BHS21,She18, Kaw19}.  These and related $C^0$-continuity results for invariants coming from various flavours of Floer theory had striking dynamical consequences, see \cite{Sey13,LRSV21,DCGHS24}. 

Ideally, one would like to prove $C^0$-continuity of $\gamma.$ This is not possible due to the fact that $\gamma$ is integer valued.  In fact, conjugation-invariant norms on $\Contoc$ are never continuous with respect to any topology, see Remark \ref{Remark:Fine_Norm}. With this in mind, we propose the following alternative to continuity.

Let $G$ be a topological group and $\| \cdot \|:G\to \R$ a conjugation-invariant norm. $\| \cdot \|$ is said to be \textit{locally bounded} if there exists a neighbourhood of $id$ on which it is bounded. One readily checks that this is equivalent to $\| \cdot \|$ being bounded in a certain neighbourhood of every element of $G.$

\begin{thm}\label{Gamma_Is_Locally_Bounded}
For $\phi,\psi \in \Contoc(\R^{2n}\times S^1)$ such that $d_{C^0}(\phi, \psi)< \frac{1}{2}$ it holds
$$|\gamma(\phi)  - \gamma(\psi)| \leq 2 .$$
In particular $\gamma$ is locally bounded on $(\Contoc (\R^{2n}\times S^1),\tau_{C^0}).$
\end{thm}

\begin{rem}
Being locally bounded depends on the choice of topology on $G.$ In the case of $\Contoc (\R^{2n}\times S^1)$, the same property holds with respect to the strong topology, since this topology is finer than $\tau_{C^0}.$ However, this is no longer true for the compact-open topology,  see Example \ref{Example:CO_vs_C0}.  
\end{rem}

\begin{rem}\label{Remark:Fine_Norm}
No genuine continuity result can be proven for any conjugation-invariant norm on $\Contoc.$ Indeed, recall that a conjugation-invariant norm $\| \cdot \|$ on a topological group $G$ is called {\it fine} if 0 is an accumulation point of $\|G\|=\{ \|g \| ~|~ g\in G \}.$ It was proven in \cite[Proposition 10]{San15}, following \cite[Theorem 1.11]{BIP08},   that for any contact manifold $(Y,\xi)$ there are no fine conjugation-invariant norms on $\Contoc(Y).$
\end{rem}

Using Theorem \ref{Gamma_Is_Locally_Bounded} and properties of $\gamma$ we show the following:

\begin{prop}\label{Lower_Semi_Gamma}
$\gamma : \Contoc(\R^{2n}\times S^1) \to \Z$ is $C^0$-lower semicontinuous. 
\end{prop}

Motivated by this result, we extend $\gamma$ to $\ContocH (\R^{2n}\times S^1)$ via limit inferior. More precisely, for any $\phi \in \ContocH (\R^{2n}\times S^1)$ we define 
$$\widetilde{\gamma}(\phi) = \max_U \min_{\psi \in U'} \gamma(\psi),$$
where $U$ is an open $C^0$-neighbourhood of $\phi$ in $\ContocH (\R^{2n}\times S^1)$ and $U' = U \cap \Contoc (\R^{2n}\times S^1).$ Theorem \ref{Gamma_Is_Locally_Bounded} guarantees that $\widetilde{\gamma}(\phi)$ is finite for all $\phi \in  \ContocH (\R^{2n}\times S^1)$ and we have the following:

\begin{thm}\label{Gamma_Homeo}
$\widetilde{\gamma}$ is a conjugation-invariant norm on $\ContocH (\R^{2n}\times S^1)$ which coincides with $\gamma$ on $\Contoc (\R^{2n}\times S^1).$ Moreover, it is locally bounded and lower semicontinuous with respect to $\tau_{C^0}.$
\end{thm}

As an immediate corollary of Theorem \ref{Gamma_Homeo} we obtain another proof that $(\ContocH(\R^{2n}\times S^1), \tau_{C^0})$ does not have Rokhlin property, see Theorem \ref{thm: Rokhlin property of the contactomorphisms of the ball}. Indeed, if it had Rokhlin property, i.e. $\conj(g) $ was dense for some $g\in  \ContocH (\R^{2n}\times S^1)$, Theorem \ref{Gamma_Homeo} would imply that for every $f\in \ContocH (\R^{2n}\times S^1)$,  $|\widetilde{\gamma}(f)-\widetilde{\gamma}(g)|$ is bounded. This is not possible since $\gamma$, and hence also $\widetilde{\gamma}$, is unbounded, see Theorem \ref{thm:Gamma_Basic}.

\begin{rem}
A conjugation-invariant norm on $\Contoc (T^*N\times S^1)$, $N$ being a closed manifold, has been defined by Zapolsky in \cite{Zapol13}. This norm is similar in spirit to $\gamma$, albeit the connecting with translated points is not as strong as in case of $\gamma.$ It would be interesting to explore if this norm is $C^0$-locally bounded.
\end{rem}

\subsection{Conjugation norm and contact fragmentation}

Assume again that $G$ is a topological group.  For $g\in G$, denote by $\conjH(g)$ the closure of $\conj(g).$ We will call $G$ {\it conjugation-decomposable} if the set
$$S(G)= \{ g\in G ~|~ id \in \conjH(g) \} $$
is a generating set for $G.$ If $G$ is conjugation-decomposable, we define the {\it conjugation norm} on $G$, denoted by $\| \cdot \|_{conj}$ as the word norm associated to $S(G).$ In other words, for $g\neq id$,
$$\| g\|_{conj} = \min \{ k~|~ \exists g_1,\ldots , g_k\in S(G),~g=g_1\ldots  g_k \}.$$
Manifestly, $\|\cdot \|_{conj}$ is integer valued and one readily checks that it is a conjugation-invariant norm on $G.$ Our first goal is to show that for every contact manifold $(Y,\xi)$, $(\Contoc(Y),d_{C^0})$ is conjugation-decomposable. This follows from the contact squeezing of balls in $\R^{2n+1}$ and the contact fragmentation lemma which we now recall. 

Let $\phi\in \Contoc(Y)$ and $\mathcal{U}= \{U_1,\ldots , U_m \}$ a finite cover of $\supp (\phi)$ by Darboux balls\footnote{By a Darboux ball we mean an image under a contact embedding of an open ball centered at 0 in $(\R^{2n+1}, \xi_0).$}. Contact fragmentation lemma, see \cite[p. 148]{Bany97}, claims that there exist $\phi_1,\ldots , \phi_N \in \Contoc(Y)$, each supported in one of $U_i$ such that $\phi = \phi_1 \circ \ldots \circ \phi_N.$ The minimal $N$ for which such decomposition of $\phi$ exists is called the {\it contact fragmentation norm with respect to} $\mathcal{U}$ and denoted by $\| \phi \|_{\mathcal{U}}.$ The {\it contact fragmentation norm of} $\phi$ is defined as
$$\| \phi \|_{frag}=\min_{\mathcal{U}} \| f \|_{\mathcal{U}},$$
where minimum runs over all covers of $\supp (\phi)$ by Darboux balls - \cite{BIP08,CS15}. Contact fragmentation norm is an example of a conjugation-invariant norm on $\Contoc(Y).$

\begin{thm}\label{Thm:Conjugation_Comparison}

Let $(Y,\xi)$ be an arbitrary contact manifold equipped with a Riemannian metric. Then $(\Contoc(Y),\tau_{C^0})$ is conjugation-decomposable and for every $\phi \in \Contoc(Y)$,  it holds
\begin{equation}\label{Inequality_Conj_Frac}
 \| \phi \|_{conj} \leq \| \phi \|_{frac}.
\end{equation}
Moreover, when $(Y,\xi)=(\R^{2n}\times S^1,\xi_0)$ we have that
\begin{equation}\label{Inequality_Sandon_Conj_Frac}
\frac{1}{2}\gamma(\phi) \leq \| \phi \|_{conj} \leq \| \phi \|_{frac}.
\end{equation}
\end{thm}

One may show that $\frac{1}{2}\gamma(\cdot) \neq  \| \cdot \|_{conj},$ see Example \ref{Example: Not_Equal} and Remark \ref{Remark:Simpler_Example}.  However, we do not know if $\| \cdot \|_{conj} \neq \| \cdot \|_{frac},$ even though it seems unlikely that the two norms are equal. 

\begin{coro}\label{Coro: Spec_Unbounded}
$\|\cdot \|_{frac}$ is unbounded on $\Contoc(\R^{2n}\times S^1)$. 
\end{coro}

Corollary \ref{Coro: Spec_Unbounded} was first proven in \cite[Corollary 1.3]{CS15}.  Another proof follows from the results of \cite{FPR18}.  Our method, which is based on $C^0$-properties of contactomorphisms, is different and independent from those two approaches.  However, Theorem \ref{Thm:Conjugation_Comparison} is related to the results of \cite{FPR18} as we will now explain.

In \cite[Theorem 3.3]{FPR18} it was proven that for every conjugation-invariant norm $\| \cdot \|$ on $\Contoc(Y)$ which is locally bounded in $C^1$-topology, it holds 
\begin{equation}\label{eq: Universal_bound_frag}
\| \cdot \| \leq C  \| \cdot \|_{frac}
\end{equation}
for some constant $C=C(\| \cdot \|).$ The $C^1$-boundedness condition for $\gamma$ is verified in \cite{CS15,FPR18} which yields Corollary \ref{Coro: Spec_Unbounded} since $\gamma (\cdot) \leq C(\gamma) \| \cdot \|_{frac}.$ With this in mind,  (\ref{Inequality_Sandon_Conj_Frac}) can be thought of as making the constant $C(\gamma)$ explicit.  On the other hand, Theorem \ref{Gamma_Is_Locally_Bounded} implies that $\gamma$ is not only $C^1$, but in fact $C^0$-locally bounded. It readily follows from the definition of $\| \cdot \|_{conj}$ that if $\| \cdot \|$ is any conjugation-invariant norm which is $C^0$-locally bounded then
\begin{equation}\label{eq: Universal_bound_conj}
\| \cdot \| \leq C \| \cdot \|_{conj} \leq C \| \cdot \|_{frac},
\end{equation} 
where $C=C(\| \cdot \|)$ is the bound of $\| \cdot \|$ in a $C^0$-neighbourhood of $id.$ Thus, under the condition of $C^0$-local boundedness,  (\ref{eq: Universal_bound_conj}) is a strengthening of (\ref{eq: Universal_bound_frag}), although, as we mentioned above, we are not aware of an example for which $\| \cdot \|_{conj} \neq \| \cdot \|_{frac}.$

\begin{rem}
Theorem \ref{Thm:Conjugation_Comparison} does not apply to $\ContocH(Y).$ This is due to the fact that contact fragmentation lemma is not known to hold for elements of $\ContocH(Y).$ It would be interesting to extend the fragmentation lemma to $\ContocH$, see \cite{Fat80,Ser24} for related results in the Hamiltonian setup.
\end{rem}

The definition of $\|\cdot \|_{conj}$ fits into a general framework of defining conjugation-invariant norms on topological groups. Namely, given a topological group $G$ and $P\subset G$ such that $P^{-1}=P$, $P$ is conjugation-invariant and $P$ generates $G$, {\it the word norm associated to $P$} is given by
\begin{equation}\label{eq:general_fragmentation}
\forall g\neq id, ~ \| g\|_{P} = \min \{ k~|~ \exists g_1,\ldots , g_k\in P,~g=g_1\ldots  g_k \}.
\end{equation}
Taking $P=S(G)$ yields $\| \cdot \|_{conj}.$ More generally, if $Q\subset G$ is any subset such that $Q^{-1}=Q$, $\conj (Q)= \{ gfg^{-1}~|~f\in Q,g\in G \} $ is conjugation-invariant and $\conj (Q)^{-1}=\conj (Q).$ Assuming that $\conj (Q)$ also generates $G$, we may apply (\ref{eq:general_fragmentation}) to $P= \conj (Q).$

Now, let $(Y,\xi)$ be a contact manifold equipped with an arbitrary Riemannian metric and denote by $\mathcal{B}_r\subset \ContocH(Y)$ the ball of radius $r$ in $d_{C^0}.$ By definition, $\mathcal{B}_r^{-1}=\mathcal{B}_r$,  and we claim that $\mathcal{B}_r$ is a generating set of $\ContocH(Y).$ Indeed, for every $\phi \in \ContocH(Y)$ there exists $\psi \in \Contoc(Y)$ such that $\phi \psi^{-1} \in \mathcal{B}_r.$ Thus, it is enough to show that $\mathcal{B}_r$ generates $\Contoc(Y)$, which can be done by cutting the isotopy $\{ \psi_t\}$ which generates $\psi \in \Contoc(Y)$ into small time intervals.

Thus, we may apply the above definition to $G=\ContocH(Y), Q=\mathcal{B}_r.$ We obtain
$$\forall \phi\neq id, ~ \| \phi\|_{r} = \min \{ k~|~ \exists \phi_1,\ldots , \phi_k,~\conj(\phi_i)\cap \mathcal{B}_r \neq \emptyset,~\phi=\phi_1\ldots  \phi_k \}.$$
It readily follows that for every $0<r_1 \leq r_2$ and every $\phi \in \ContocH(Y)$
$$\| \phi \|_{r_2} \leq \| \phi \|_{r_1} \leq  \| \phi \|_{conj}.$$
In this sense, $\| \cdot \|_{conj}$ can be though of as a limit as $r\to 0$ of $\|\cdot \|_{r}$ and we ask the following question.

\begin{question}
Is it true that for every $(Y,\xi)$ and every $\phi \in \ContocH(Y)$ it holds $\lim_{r\to 0^+} \| \phi \|_r=\| \phi \|_{conj}$?
\end{question}

\subsection{Quantitative weak conjugacy}\label{Section:Intro_Classes}

In order to quantify the failure of the Rokhlin property we introduce the following notion. Let $G$ be a topological group, $k\geq 1$ an integer and for $h\in G$ denote by $\conjH(h)$ the closure of $\conj(h).$ We say that $f,g \in G$ are {\it $k$-conjugation connected} if there exist $\phi_0, \ldots ,\phi_k \in G$ such that
\begin{enumerate}
\item[1)] $\phi_0=f$, $\phi_k=g$;
\item[2)] For each $i=0,\ldots , k-1$ there exists $\psi_i\in G$ such that $\phi_i ,\phi_{i+1} \in \conjH(\psi_i).$
\end{enumerate}
We now define, for $f\neq g$,
$$d_{cc}(f,g)=\min \{ k~ |~ f \text{ and } g \text{ are } k\text{-conjugation connected}  \}$$
and $d_{cc}(g,g)=0$ for all $g\in G.$ If $f$ and $g$ are not $k$-conjugation connected for any $k$ we set $d_{cc}(f,g)=+\infty.$

One readily checks that $d_{cc}$ is an extended metric on $G$, i.e. a metric which might take the value $+\infty.$ In general, there is no reason for $d_{cc}$ to be bi-invariant and in fact it is neither left nor right invariant for $G=(\ContocH(\R^{2n}\times S^1),\tau_{C^0})$, see Example \ref{Example_DCC}. If $G$ has Rokhlin property then $d_{cc}$ is a trivial metric.\footnote{We call a metric $d$ {\it trivial} if $d(x,y)=1$ if and only if $x \neq y.$ Such a metric is sometimes called {\it discrete}, however we choose not to use this term in order to avoid confusion with a metric which takes a discrete set of values.} Thus, $d_{cc}$ measures how much the group fails to have Rokhlin property. 

The notion of conjugation connectedness is related to the notion of weak conjugacy introduced in \cite[Definition 51]{LRSV21}. Recall that two elements $f,g$ of a topological group $G$ are said to be \textit{weakly conjugate} if for any conjugation-invariant continuous map $F:G\to X$ to a Hausdorff topological space $X$, $F(f)=F(g).$ One readily checks that if $f$ and $g$ are $k$-conjugation connected for any $k\geq 1$, then they are also weakly conjugate. In fact, in \cite[p. 2716]{LRSV21} a condition very close to conjugation connectedness is introduced as a criterion for weak conjugacy.  $d_{cc}$ can be thought of as quantification of this criterion.

\begin{thm}\label{thm: Rokhlin property of the contactomorphisms of B2nS1}
For $(\ContocH(\R^{2n}\times S^1),\tau_{C^0})$, $d_{cc}$ is unbounded.  
\end{thm}

Theorem \ref{thm: Rokhlin property of the contactomorphisms of B2nS1} is a direct consequence of Theorem \ref{Gamma_Homeo} and unboundedness of $\gamma$, see Section \ref{Sec:Proof_cc} for details. It readily implies the second part of Theorem \ref{thm: Rokhlin property of the contactomorphisms of the ball} as $d_{cc}$ is not trivial for $(\ContocH(\R^{2n}\times S^1),\tau_{C^0}).$ Moreover, it gives a better understanding of the conjugacy classes in $\ContocH(\R^{2n}\times S^1)$, since the argument based on non-squeezing, which we provide in the proof of Theorem \ref{thm: Rokhlin property of the contactomorphisms of the ball}, does not even show that $d_{cc}$ is not a trivial metric. That being said, our understanding of $d_{cc}$ is still quite limited. For example, we do not know the answer to the following basic question:

\begin{question}
Are any two elements of $(\ContocH(\R^{2n}\times S^1),\tau_{C^0})$ $k$-conjugation connected for some finite $k$? In other words, is $d_{cc}$ a genuine metric?
\end{question}

\subsection{Prequantization spaces}\label{Sec: Prequantization spaces}

Results from Sections \ref{Section:Intro_Rokhlin} - \ref{Section:Intro_Classes} partially extend to more general prequantization spaces.  Let $(W,d\lambda)$ be a connected exact symplectic manifold.{\it The prequantization space} of $W$ is a contact manifold $(W\times S^1, \alpha=\lambda+d\theta)$, where $\theta$ is a coordinate on $S^1$ induced from the quotient $S^1=\R / \Z.$

From now on, we assume that $W$ is a Liouville manifold, i.e. a completion of a Liouville domain, such that\footnote{The condition on the first Chern class is assumed in accordance with \cite{AM18}. It seems likely that it can be dropped.} $c_1|_{\pi_2(W)}=0$, $g$ is a Riemannian metric on $W$ and $W\times S^1$ is equipped with a product metric $g \oplus g_{std}$.  We denote by $SH(W)$ the symplectic homology of $W.$ The following result is a generalization\footnote{To be precise, in order for Theorem \ref{thm:Rok_Preq} to generalize Theorem \ref{thm: Rokhlin property of the contactomorphisms of the ball}, we should take the primitive $\frac{1}{2}(\sum_{i=1}^n x_i dy_i - y_i d x_i)$ for $\omega_0$ on $\R^{2n}.$ Since the proof of both theorems relies on contact non-squeezing, the choice of a primitive is irrelevant, see Remark \ref{Remark_Change_of_Basis}.} of the second part of Theorem \ref{thm: Rokhlin property of the contactomorphisms of the ball}.

\begin{thm}\label{thm:Rok_Preq}
If $SH(W)=0$ then $(\ContocH(W \times S^1), \tau_{C^0})$ does not have Rokhlin property.
\end{thm}

The proof of Theorem \ref{thm:Rok_Preq} is a minor modification of the proof of the second part of Theorem \ref{thm: Rokhlin property of the contactomorphisms of the ball}, relaying on the non-squeezing result, Theorem \ref{Thm:Non-squeezing_AM}, proven by Albers and Merry in \cite{AM18}.  The $SH(W)=0$ condition is used to guarantee finiteness of the spectral symplectic capacities of the domains in $W$, following \cite{BK22}. 

Building on the same idea, only using a different non-squeezing result from \cite{AM18}, we may prove another theorem in the same spirit. Namely,  let $m\geq 1$ and equip $W \times \R^{2m} \times S^1 $ with the contact form $ \lambda- \sum_{i=1}^n y_i dx_i + d\theta$ and with the sum metric. Denote by $c_{HZ}$ the Hofer-Zehnder capacity of subsets in $W \times \R^{2m} .$

\begin{thm}\label{thm:Roklin_Stabilization}
Assume that $W$ is a completion of a Liouville domain $W_0$ such that $c_{HZ}(W_0)$ is finite.  Then $(\ContocH(W \times \R^{2m} \times S^1), \tau_{C^0})$ does not have Rokhlin property.
\end{thm}

\begin{rem}
It was shown in \cite{Mailhot25}, that the vanishing of symplectic homology is in fact equivalent to the finiteness of the spectral symplectic capacity. Finiteness of the Hofer-Zehnder capacity is a strictly weaker condition. Indeed, $SH(W)=0$ implies that $c_{HZ}(W_0)$ is finite by \cite[Theorem 1.3]{BK22} and \cite[Corollary 8.3]{FS07}. On the other hand \cite{Irie14} provides examples of unit codisc bundles (whose symplectic homology is never vanishing) with the finite Hofer-Zehnder capacity. 
\end{rem}

Using Rabinowitz Floer homology, Albers and Merry defined spectral invariants $c_{AM}:\ContocU(W\times S^1)\to \R$ on the universal cover of $\Contoc(W\times S^1).$ While these invariants share a lot of properties with Sandon's spectral invariants, \cite{AM18} does not define\footnote{To the best of our understand this is due to difficulties related to triangle inequality stemming from the difficulties related to product structures in Rabinowitz Floer homology.} the spectral norm on $\Contoc(W\times S^1).$ The next theorem shows $C^0$-continuity of $c_{AM}.$

\begin{thm}\label{Thm: AM_Continuity}
Let $\widetilde{\phi}\in \ContocU(W\times S^1)$ and $\phi=\widetilde{\phi}_1 \in \Contoc(W\times S^1).$ If $\| \phi \|_{C^0}<\frac{1}{2}$ then $|c_{AM}(\widetilde{\phi})| \leq \| \phi \|_{C^0}.$
\end{thm}

Let us mention that our proof of Theorem \ref{Gamma_Is_Locally_Bounded} goes through and analogue of Theorem \ref{Thm: AM_Continuity} for Sandon's spectral invariants, see Section \ref{Section:Gamma_Proof}. Moreover, a natural candidate for the spectral norm $\gamma_{AM}:\Contoc(W\times S^1)\to \Z$ can be defined, see Section \ref{Section: Albers_Merry_Basics}, and Theorem \ref{Thm: AM_Continuity} implies that $\gamma_{AM}$ is $C^0$-locally bounded. If one was to show that $\gamma_{AM}$ is in fact a conjugation-invariant norm on $\Contoc(W\times S^1)$, the results from Sections \ref{Section:Intro_Rokhlin} - \ref{Section:Intro_Classes} would readily generalize. To this end, we put forward the following conjecture:

\begin{conjecture}\label{Conjecture_Preq}
For $(\ContocH(W \times S^1), \tau_{C^0})$, $d_{cc}$ is unbounded.
\end{conjecture}

\begin{rem}
In this paper, we only considered the continuity of spectral invariants in the case of non-compact contact manifolds. In the case of closed contact manifolds, a number of different constructions of Floer-theoretic spectral invariants which satisfy triangle inequality, have recently appeared in \cite{DUZ23,Cant24,ASZ25}.  It would be interesting to explore the $C^0$-continuity of these invariants.
\end{rem}

Lastly, we note that $W\times \R$ can also be equipped with a contact form given by $\alpha =dz+ \lambda$, $z$ being a coordinate on $\R.$ Now, for $W=\R^{2n}$ we recover $(\R^{2n+1},\xi_0)$ discussed in Theorem \ref{thm: Rokhlin property of the contactomorphisms of the ball}. In analogy with this result, we ask the following question:

\begin{question}
For which exact $W$ is $d_{cc}$ bounded for $(\ContocH(W \times \R), \tau_{C^0})$?
\end{question}

\subsection*{Acknowledgements}

We are very grateful to Marco Mazzucchelli and Sobhan Seyfaddini for suggesting this project to us, as well as for a number of fruitful ideas and very useful discussion in the early stages. Their help has been truly indispensable. We thank Vincent Humili\`{e}re and Leonid Polterovich for illuminating comments and Frol Zapolsky for discussions related to \cite{Zapol13}. We also thank participants of the symplectic seminars in Belgrade and Montr\'{e}al for numerous questions which helped us improve exposition. During the writing of the article, B.S. was supported by the ERC Starting Grant 851701 and V.S. was supported by the ERC Starting Grant 851701, CRM-ISM postdoctoral fellowship and Fondation Courtois.


\section{Preliminaries}

\subsection{Different topologies on homeomorphism groups}\label{Section_Maps_Topologies}

We review some facts about different topologies on the spaces of continuous maps, which will be needed in the paper. Most of the material in this section is standard and treated in detail in \cite[Chapter 7]{Mun00}.  We provide proofs for the statements we were not able to locate in the literature.

Let $X$ be a Hausdorff topological space, $(Y,d)$ a metric space and $C(X,Y)$ the space of continuous maps from $X$ to $Y.$ The {\it uniform} or {\it $C^0$-distance}\footnote{To be precise,  $d_{C^0}$ is an extended distance, i.e. it may take value $+\infty.$ This is rather irrelevant since topology only depends on arbitrarily small balls. Moreover, in the rest of the paper $d_{C^0}$ will always be considered on compactly supported homeomorphisms and will thus only take finite values. } between $f,g \in C(X,Y)$ is given by
$$d_{C^0}(f,g)= \sup_{x\in X} d(f(x),g(x)).$$
The topology on $C(X,Y)$ induced by $d_{C^0}$ is called the {\it uniform} or {\it $C^0$-topology} and is denoted by $\tau_{C^0}.$ 

A different topology on $C(X,Y)$ is defined as follows. For $f\in C(X,Y)$, $K\subset X$ compact and $\varepsilon>0$, denote by
$$\mathcal{U}_{f,K,\varepsilon} = \{g \in C(X,Y)~|~ d_{C^0}(f|_K, g|_K)<\varepsilon \}.$$
The collection of all such $\mathcal{U}_{f,K,\varepsilon}$ is a basis of the {\it compact-open topology} on $C(X,Y).$ We denote this topology by $\tau_{CO}.$

Lastly, for $f\in C(X,Y)$ and a continuous function $\delta:X\to (0,+\infty)$ let
$$\mathcal{V}_{f,\delta} = \{g \in C(X,Y)~|~ (\forall x \in X) ~ d(f(x), g(x))<\delta(x) \}.$$
The collection of all such $\mathcal{V}_{f,\delta}$ is a basis of the  {\it strong} or {\it fine topology} on $C(X,Y).$ We denote this topology by $\tau_{s}.$ It is clear from the definitions that
$$\tau_{CO} \subset \tau_{C^0} \subset \tau_{s}.$$

When $X$ is compact, all three of the above topologies coincide. This is no longer true in the cases of interest for us. We will use the following criteria for convergence in different topologies.

\begin{lemma}\label{Lemma:CO_Convergence}
A sequence $f_i\in C(X,Y), i\geq 1$ converges to $f\in C(X,Y)$ in $\tau_{CO}$ if and only if for every compact $K\subset X$, $f_i|_K$ converges to $f|_K$ uniformly. 
\end{lemma}

\begin{lemma}\label{Lemma:Strong_Convergence}
Let $f_i\in C(X,Y), i\geq 1$ be a sequence. Assume that there exists a compact $K\subset X$ such that all $f_i$ coincide outside of $K$, and assume that $f_i$ converges to $f$ in $\tau_{C^0}.$ Then $f_i$ converges to $f$ in $\tau_{s}$ as well. 
\end{lemma}
\begin{proof}
We show that for every basis element $\mathcal{V}_{f,\delta}$, there exists $i_0$ such that $f_i\in \mathcal{V}_{f,\delta}$ for all $i\geq i_0.$ Indeed, since $\delta>0$ is continuous, we have that $\min_K \delta>0.$ Taking $i_0$ such that $d_{C^0}(f_i,f)<\min_K \delta$ for all $i\geq i_0$ finishes the proof.
\end{proof}

Let us mention that compact-open and strong topologies do not depend on the metric on $Y,$ but only on the underlying topology. This is not the case with the $C^0$-topology, as can be seen from elementary examples.

Throughout the paper, we will need the following auxiliary lemma.

\begin{lemma}\label{Lemma:Top_Group}
Let $(X,d)$ be a metric space. Then $(\Homeo_c(X),\tau_{C^0})$ is a topological group.
\end{lemma}
\begin{proof}
Let $\{ \phi_n \}, \{ \psi_n \}, n\geq 1$ be two sequences in $\Homeo_c(X)$ such that $\phi_n \to \phi$ and $\psi_n \to \psi$ in $\tau_{C^0}.$ Triangle inequality gives us
$$d_{C^0}(\phi\circ \psi, \phi_n\circ \psi_n)\leq d_{C^0}(\phi\circ \psi, \phi\circ \psi_n) + d_{C^0}(\phi\circ \psi_n, \phi_n\circ \psi_n).$$
Since $\phi$ is compactly supported it is uniformly continuous and thus $d_{C^0}(\phi\circ \psi, \phi\circ \psi_n)\to 0$ as $n\to \infty.$ On the other hand $d_{C^0}(\phi\circ \psi_n, \phi_n\circ \psi_n)=d_{C^0}(\phi, \phi_n)\to 0$ as $n\to \infty$, which proves that composition is continuous.

To prove that taking inverses is continuous we observe that
$$d_{C^0}(\phi^{-1}, \phi_n^{-1})=d_{C^0}(\phi^{-1} \circ \phi_n, id)$$
and the proof follows from the continuity of the composition.
\end{proof}

Lastly, we prove the following proposition.

\begin{prop}\label{Prop:separable}
 $(\Contoc(\R^{2n+1}), \tau_{C^0})$ is separable.
\end{prop}
\begin{proof}
For an integer $k\geq 1$, denote by $B(k)\subset \R^{2n+1}$ the open ball of radius $k$ with center at the origin. Since $\Contoc(\R^{2n+1}) = \cup_{k_\geq 1} \Contoc(B(k))$, it is enough to show that $\Contoc(B(k), \tau_{C^0})$ is separable for every $k.$ To this end, recall that a subset of a separable metric space is separable. Now, we use that $\Contoc(B(k))\subset C(\overline{B}(k), \R^{2n+1})$ where $\overline{B}(k)$ denotes the closed ball. To show that $(C(\overline{B}(k), \R^{2n+1}), \tau_{C^0})$ is separable, we notice that maps of the form $(p_1,\ldots, p_{2n+1})$, $p_i$ being polynomials with rational coefficients, constitute a countable dense subset of $C(\overline{B}(k), \R^{2n+1})$ by the real version of the Stone-Weierstrass theorem.
\end{proof}

\subsection{Contact squeezing} Let $(Y,\xi)$ be a contact manifold and $\alpha$ be a contact form adapted to the contact structure, that is $\xi=\ker(\alpha)$. Denote by $R_\alpha$ the Reeb vector field associated to $\alpha$, given by $d\alpha (R_\alpha, \cdot)=0,\alpha(R_\alpha)=1.$ To any time-dependent function $H_t$ on $Y$ we can associate a unique time-dependent vector field $X_{H_t}$ given by the following equations:

\begin{align}\label{form: generating a contact isotopy}
\begin{cases}
d\alpha(X_{H_t}, \cdot)=\mathrm{d}H_t(R_\alpha)\alpha-dH_t\\
\alpha(X_{H_t})=H_t
\end{cases}
\end{align}

The flow of the vector field $X_{H_t}$ defines a contact isotopy $\phi_t$, $\phi_0=id$ and $H_t$ is called \textit{the contact Hamiltonian of} $\phi_t.$ Conversely, every contact isotopy $\phi_t$, $\phi_0=id$ is generated by a contact Hamiltonian given by $H_t(\phi_t(x))=\alpha (\dot{\phi}_t(x)).$

We use contact Hamiltonians to define a squeezing contactomorphism, which will be used throughout the paper. More precisely, we have the following lemma:

\begin{lemma}\label{Lemma:Ball_Sqeezing}
Denote by $B(r)\subset \R^{2n+1}$ the ball of radius $r.$ Let $0<r<R$ and $0<a \leq 1.$ There exists $\Psi_a \in \Contoc (B(R), \xi_0)$ such that $\Psi_a(x,y,z)=(a x, a y, a^2 z)$ for every $(x,y,z)\in B(r).$
\end{lemma}
\begin{proof}
First, we notice that for every $a>0$, the map $\widetilde{\Psi}_a (x,y,z)=(a x, a y, a^2 z)$ is a (non-compactly supported) contactomorphism of $(\R^{2n+1},\xi_0).$ Indeed, $\widetilde{\Psi}_a ^* \alpha_0 = a^2 \alpha_0.$ Moreover, $\widetilde{\Psi}_a $ is isotopic to $id$ through contactomorphisms, via the isotopy $\widetilde{\Psi}_{a,t} = \widetilde{\Psi}_{1+(a-1)t}.$ Thus, there exists a contact Hamiltonian $\widetilde{H}:\R^{2n+1}\times [0,1] \to \R$ which generates  $\widetilde{\Psi}_{a,t}.$ Let $\rho:\R^{2n+1}\to \R$ be a cut-off function, equal to 1 on $B(r)$ and equal to 0 outside of $B(R).$ By definition, Hamiltonian $H=\rho \widetilde{H}$ generates an isotopy which on $B(r)$ coincides with $\widetilde{\Psi}_{a,t}$, while it is supported inside $B^{2n+1}(R).$ $\Psi_a$ is the time-one map of this isotopy.
\end{proof}

\subsection{Contact non-squeezing}
In \cite{EKP06} Eliashberg, Kim and Polterovich introduced the following notion.

\begin{dfn}
Let $(Y,\xi)$ be a contact manifold and $U,V \subset Y$ open. We say that $V$ can be squeezed into $U$ if there exists $\phi \in \Contoc(Y)$ such that $\phi (\overline{V})\subset U.$
\end{dfn}

Let $B(r)\subset \R^{2n}$ be the ball of radius $r.$ We consider domains in $(\R^{2n} \times S^1,\xi_0)$ of the form $B(r) \times S^1.$

\begin{thm}\label{Thm:Non-squeezing_Original}
If $1\leq \pi r^2 <\pi R^2$ then $B(R)\times S^1$ cannot be squeezed into $B(r)\times S^1.$
\end{thm}

Theorem \ref{Thm:Non-squeezing_Original} was originally proven in \cite{EKP06} under an additional assumption that $[\pi r^2, \pi R^2]$ contains an integer. In stated generality, it was proven in \cite{Chiu17}, see also \cite{Fraser16}. For the proofs using generating functions see \cite{San1,FSZ23}. It is worth pointing out that if $\pi r^2 < \pi R^2 <1$ then $B(R)$ can be squeezed into $B(r)\times S^1$, see \cite{EKP06} for details.

\begin{rem}\label{Remark_Change_of_Basis}
Theorem \ref{Thm:Non-squeezing_Original} also holds if instead of $\xi_0$ we consider the contact structure $\xi_0'=\ker \alpha_0'$, where $\alpha_0' = dz+ \frac{1}{2} (\sum_{i=1}^n x_i dy_i - y_idx_i).$ Indeed, the map $\Psi: (\R^{2n}\times S^1, \xi_0) \to (\R^{2n}\times S^1,\xi_0')$ given by 
$$\Psi({\bf x},{\bf y},z)=\left( \frac{{\bf x}-{\bf y}}{\sqrt{2}}, \frac{{\bf x}+{\bf y}}{\sqrt{2}}, z - \frac{{\bf x} \cdot {\bf y}}{2} \mod 1 \right),$$
is a contactomorphism which maps $B(r)\times S^1$ to $B(r)\times S^1$ for all $r>0.$
\end{rem}

The original non-squeezing result of \cite{EKP06} was generalized by Albers and Merry in \cite{AM18}. We now recall their result.

Let $(W,d\lambda)$ be a Liouville manifold (completion of a Liouville domain), $c_1|_{\pi_2(W)}=0$.Denote by $\mathcal{L}_t:W\to W$ the Liouville flow given by the vector field $d\lambda(X_{\mathcal{L}}, \cdot)=\lambda$ and for $r>0,q\in W$, we denote $r q = \mathcal{L}_{\log r }(q).$ Given $U\subset W$ open, let $c_{sp}(U)$ be the symplectic capacity of $U$, see \cite[Section 5]{AM18} and references therein. We will not elaborate on the definition and properties of $c_{sp}.$ We only note that $c_{sp}(U)>0$ if $U\neq \emptyset$ and that $c_{sp}(r U)=r c_{sp}(U)$ for all $r>0.$

\begin{thm}[\cite{AM18}]\label{Thm:Non-squeezing_AM}
Let $U\subset W$ be an open set with compact closure such that $c_{sp}(U)=1.$ Let $0<r<R$ and assume that $[r,R]$ contains an integer. Then $(R U) \times S^1$ cannot be squeezed into $(rU)\times S^1.$
\end{thm}

Lastly, we recall Theorem 1.24 from \cite{AM18}. It applies to the contact manifold $W\times \R^{2m}\times S^1$ equipped with the contact structure $\ker (\lambda+\lambda_0 + d\theta).$

\begin{thm}[\cite{AM18}]\label{Thm:Non-squeezing_Stabilization}
Let $U\subset W$ be an open set with compact closure such that $c_{HZ}(U)$ is finite. If $r_0,r_1>0$ satisfy $\lceil \pi r_0^2 \rceil < \lceil c_{HZ}(U) \rceil$ and $r_1\geq \sqrt{\frac{c_{HZ}(U)}{\pi}} +1$ then $U\times B(r_1) \times S^1$ cannot be squeezed into $U\times B(r_0) \times S^1.$
\end{thm}

\subsection{Contact action spectrum} 

Let $(W,d\lambda)$ be an exact symplectic manifold and $(W\times S^1, \alpha = \lambda+d\theta)$ its prequantization space. Denote by $\Phi^\alpha$ the Reeb flow of $\alpha.$ One readily checks that $\Phi^\alpha_t(q,\theta)=(q, \theta + t \mod 1).$  $z \in W\times S^1$ is called a {\it translated point} of $\phi\in \Contoc ( W\times S^1)$ if there exists $\tau\in \R$ such that $\Phi^\alpha_{-\tau} ( \phi (z))=z$ and $(\phi^* \alpha)_z=\alpha_z.$ {\it The action of} $z$ is given by $\mathcal{A}_\alpha(z) = \int_0^1 d\theta(\dot{\phi}_t(z)) dt,$ where $\{ \phi_t \}_{t\in [0,1]} $ is a smooth path in $\Contoc ( W\times S^1)$ such that $\phi_0=id$ and $\phi_1=\phi.$

\begin{lemma}\label{Lemma:TP_Action} The action of a translated point $z$ of $\phi \in \Contoc ( W\times S^1)$ does not depend on the choice of a path $\{ \phi_t \}_{t\in [0,1]}.$
\end{lemma}

\begin{proof}
Let $\{ \phi_t \}_{t\in [0,1]}, \{ \psi_t \}_{t\in [0,1]}$ be two smooth paths in $\Contoc ( W\times S^1)$ starting at $id$ and ending at $\phi.$ Denote by $\psi^{op} \# \phi$ the concatenation given by
$$(\psi^{op} \# \phi)_t= \begin{cases}
      \phi_{2t}, & \text{for}\ 0\leq t \leq 1/2 \\
      \psi_{2-2t}, & \text{for}\ 1/2 \leq t \leq 1 \\
      \end{cases}.$$
Let $x\in W\times S^1$ be an arbitrary point outside of the support of $\{ \phi_t \}_{t\in [0,1]}, \{ \psi_t \}_{t\in [0,1]},$ and $l:[0,1]\to W\times S^1$ an arbitrary smooth path such that $l(0)=x,l(1)=z.$ Define a maps $\Psi:[0,1] \times [0,1] \to W \times S^1$ as
$$\Psi(s,t)= (\psi^{op} \# \phi)_t (l(s)). $$
$\Psi$ is a contraction of the loop $\psi^{op} \# \phi$ along $l$, i.e.
$$\Psi(0,t)=x, ~\Psi(1,t)=(\psi^{op} \# \phi)_t(z), ~ \Psi(s,0)=\Psi(s,1)=l(s).$$
Thus, applying Stokes' theorem we get that
$$0 = \int\limits_{[0,1]\times [0,1]} \Psi^* d(d\theta)= \int_0^1 d\theta \left( \frac{d}{dt} (\psi^{op} \# \phi)_t(z) \right) dt =$$
$$ =2\left( \int_0^1 d\theta (\dot{\phi}_t(z) )dt - \int_0^1 d \theta (\dot{\psi}_t(z)) dt \right) .$$
which proves the claim.
\end{proof}

In light of Lemma \ref{Lemma:TP_Action}, we define the {\it contact action spectrum} of $\phi \in \Contoc ( W\times S^1)$ as 
$$Spec_\alpha(\phi)= \{ \mathcal{A}_\alpha(z)~|~z \text{ is a translated point of }\phi \}.$$

\subsection{Viterbo's spectral invariants and spectral norm}\label{Subsection:Viterbo_Basics}

We will present certain properties of Viterbo's spectral invariants which will be used in the paper. For a detailed account, see \cite{Vit92, San15}.

Let $\omega_0=\sum_{i=1}^n dx_i \wedge dy_i=d \lambda_0$, $\lambda_0 = -\sum_{i=1}^n y_i dx_i$ be the canonical symplectic form of $\R^{2n}.$ Given a compactly supported Hamiltonian $H_t:\R^{2n}\times [0,1] \to \R$, the Hamiltonian vector field defined by $H$ is given by $\omega(X_H,\cdot) = - dH_t.$ We denote by $\phi^H_t, t\in [0,1]$ the flow of $X_H$ and by $\Ham_c(\R^{2n})$ the group of compactly supported Hamiltonian diffeomorphisms of $\R^{2n}.$ The action of a path $l : [0,1] \to \R^{2n}$ is defined as
$$\mathcal{A}_H (l)= \int_0^1 H_t(l(t))dt -\int_0^1 \lambda_0 (\dot{l}(t))dt.$$
Taking the action of a path $l(t) = \phi^H_t(q) $, we define the action of  a point $q\in \R^{2n}$ with respect to $H$ as
\begin{equation}\label{eq:Action_Points}
A_{\phi}(q)=\mathcal{A}_H (\{ \phi^H_t(q) \})= \int_0^1 H_t(\phi^H_t(q))dt -\int_0^1 \lambda_0 (\dot{\phi}^H_t(q))dt.
\end{equation}
The action only depends on $\phi=\phi_1$ and not on the isotopy $\{ \phi_t \}$, which justifies the above notation.  Given $\phi \in \Ham_c(\R^{2n})$, the action spectrum of $\phi$ is
$$Spec(\phi)= \{ A_\phi(q)~|~\phi(q)=q \}.$$
In \cite{Vit92} Viterbo defined spectral invariants $c^-_{Ham},c^+_{Ham}: \Ham_c(\R^{2n})\to \R$ whose properties are summarized by the following theorem:

\begin{thm}\label{Thm_Viterbo_Basic}
For all $\phi,\psi \in \Ham_c(\R^{2n}),$ $c^{\pm}_{Ham}$ satisfy
\begin{itemize}
\item[1)] $c^{\pm}_{Ham} \in Spec(\phi) ;$
\item[2)] $c^-_{Ham}(\phi)\leq 0 \leq c^+_{Ham}(\phi)$ and $c^-_{Ham}(\phi)= c^+_{Ham}(\phi)=0$ if and only if $\phi=id$;
\item[3)] $c^-_{Ham}(\phi)=-c^+_{Ham}(\phi^{-1}) ;$
\item[4)] $c^+_{Ham}(\phi \psi) \leq c^+_{Ham}(\phi) + c^+_{Ham}(\psi);$
\item[5)] $c^{\pm}_{Ham}(\psi \phi \psi^{-1})=c^{\pm}_{Ham}(\phi) ;$
\item[6)] $c^{\pm}_{Ham}$ are continuous with respect to Hofer's metric.
\end{itemize}
\end{thm}

\begin{dfn}
The spectral norm $\gamma:\Ham_c (\R^{2n}) \to \R$ is defined as 
$$\gamma(\phi)=c^+(\phi) -  c^-(\phi).$$
\end{dfn}

\subsection{Sandon's spectral invariants and spectral norm}\label{Subsection:Sandon_Basics}

We briefly review basic properties of spectral invariants and spectral norm defined by Sandon. For a detailed treatment we refer the reader to \cite{San1,San2,San15}.  We also prove Lemma \ref{Lemma:Lifts_Computations} which will be useful when constructing examples in Section \ref{Sec:Examples}.

In order to define the spectral norm, we first need to consider spectral invariants $c^-,c^+:\Contoc(\R^{2n}\times S^1) \to \R.$ These are defined using the mix-max procedure for filtered homology of generating functions.  

\begin{prop}\label{prop:C_non_integer} For all $\phi \in \Contoc(\R^{2n}\times S^1)$, $c^-,c^+$ satisfy
\begin{itemize}
\item[1)] $c^-(\phi)\leq 0 \leq c^+(\phi)$ and $c^-(\phi)= c^+(\phi)=0$ if and only if $\phi=id$;
\item[2)] $c^-(\phi),c^+(\phi) \in Spec_{\alpha_0}(\phi) .$
\end{itemize}
\end{prop}

The key deficiency of $c^-,c^+$ lies in the fact that they are not conjugation-invariant. This is due to the fact that translated points do not persist under cojugation. However, if the action of a translated point is an integer, by definition this translated point is a fixed point (note that by convention $S^1$ has length 1). Fixed points do persists under conjugation, namely if $x$ is a fixed point of $\phi$ then $\psi(x)$ is a fixed point of $\psi \phi \psi^{-1}.$ Elaborating on these observations one comes to the conclusion that, while conjugation might change $c^{\pm}$, it still holds that
$$\lfloor c^\pm(\phi) \rfloor \leq c^\pm(\psi\phi \psi^{-1})  \leq  \lceil c^\pm(\phi) \rceil ,$$
where $\lfloor \cdot \rfloor , \lceil \cdot \rceil $ denote the lower and upper integer parts.  Thus, insted of $c^-,c+$, we wish to use invariants $\lfloor c^-(\cdot) \rfloor , \lceil c^+(\cdot) \rceil .$ Their properties are summarized in the following proposition.

\begin{prop}\label{prop: C_integer} For all $\phi,\psi \in \Contoc(\R^{2n}\times S^1)$ it holds
\begin{itemize}
\item[1)] $\lfloor c^-(\psi \phi \psi^{-1}) \rfloor=\lfloor c^-(\phi) \rfloor$ and $\lceil c^+(\psi \phi \psi^{-1}) \rceil= \lceil c^+(\phi) \rceil$;
\item[2)] $\lfloor c^-(\phi^{-1}) \rfloor= -\lceil c^+(\phi) \rceil$;
\item[3)] $\lceil c^+(\psi) \rceil \geq \lceil c^+(\phi) - c^+(\phi \psi^{-1}) \rceil \geq \lceil c^+(\phi)\rceil - \lceil c^+(\phi \psi^{-1}) \rceil $.
\end{itemize}
\end{prop}

\begin{dfn}
The spectral norm $\gamma:\Contoc(\R^{2n}\times S^1) \to \Z$ is defined as 
$$\gamma(\phi)=\lceil c^+(\phi) \rceil - \lfloor c^-(\phi) \rfloor.$$
\end{dfn}

\begin{thm}\label{thm:Gamma_Basic}
$\gamma$ is a conjugation-invariant norm on $\Contoc(\R^{2n}\times S^1).$ Moreover, it is unbounded.
\end{thm}
The first part of Theorem \ref{thm:Gamma_Basic} is a direct consequence of Propositions \ref{prop:C_non_integer} and \ref{prop: C_integer}. The proof that $\gamma$ is unbounded can be found in \cite[Section 6]{San2}.

Lastly, let us show how to compute invariants $c^-,c^+$ for a very special class of contactomorphisms. Let $\phi \in \Ham_c(\R^{2n}).$\textit{ The lift of} $\phi$ is a contactomorphism $\widehat{\phi}\in \Contoc$ defined by
\begin{equation}\label{Equation:Lift}
\widehat{\phi}(q,z)=(\phi(z),z+A_\phi(q) \mod 1),
\end{equation}
where $A_\phi(q)$ is given by (\ref{eq:Action_Points}). It is generated by the contact Hamiltonian $\widehat{H}_t(q,z)=H_t(q).$ In \cite[Proposition 3.18]{San1} it was proven that
\begin{equation}\label{eq:Spec_Lifts}
c^{\pm}(\widehat{\phi} )= c^{\pm}_{Ham} (\phi).
\end{equation}

We will need the following lemma:

\begin{lemma}\label{Lemma:Lifts_Computations}
Let $H:\R^{2n} \to \R$ be an autonomous compactly supported Hamiltonian and $\phi \in \Ham_c(\R^{2n})$ the time-1 map generated by $H.$ If $H$ is sufficiently $C^2$-small then
$$c^-(\widehat{\phi})=\min H,~ c^+(\widehat{\phi})=\max H.$$
\end{lemma}

\begin{proof}
 By (\ref{eq:Spec_Lifts}) it is enough to show that $c^-_{Ham}(\phi)=\min H$, $c^+_{Ham}(\phi)=\max H.$ For sufficiently $C^2$-small $H$, $A_\phi : \R^{2n}\to \R$ is a compactly supported generating function for $\phi$ without ghost variables, see \cite[Lemma 9.1.4]{McDS17}.  Since there are no ghost variables, by definition $c^-_{Ham}(\phi)=\min A_\phi$ and $c^+_{Ham}(\phi)=\max A_\phi.$ On the other hand, $\min A_\phi$ and $\max A_\phi$ are minimal and maximal actions of fixed points of $\phi.$ Finally, for sufficiently $C^2$-small $H$, the only fixed points of $\phi$ are critical points of $H$. Each such fixed point is fixed by the whole isotopy $\phi_t$ and its action is equal to the value of $H$ at that point. The claim follows.
\end{proof}

\subsection{Spectral invariant of Albers and Merry}\label{Section: Albers_Merry_Basics}

Let $(W,d\lambda)$ be a Liouville manifold such that $c_1(W)=0$ and $W\times S^1$ its prequantization space. In \cite{AM18} Albers and Merry defined a spectral invariant $c_{AM}:\ContocU(W\times S^1)\to \R$ (in notation of \cite{AM18}, $c_{AM}(\cdot)=-c(\cdot)=-c(\cdot,\mu_\Sigma)$). This invariant is analogous to $c^+$ invariant of Sandon, see Section \ref{Subsection:Sandon_Basics}, and it has similar properties, see \cite{AM18} for details. A natural candidate for the spectral norm on $\ContocU(W\times S^1)$ is given by
$$\gamma_{AM}'(\widetilde{\phi})= | \lceil c_{AM}(\widetilde{\phi}) \rceil   | + | \lceil c_{AM}(\widetilde{\phi}^{-1}) \rceil  |,$$
while a natural candidate\footnote{In fact, $\gamma_{AM}'$ and $\gamma_{AM}$ have been defined in the original preprint of \cite{AM18}.} for the spectral norm on $\Contoc(W\times S^1)$ is given by
$$\gamma_{AM}(\phi)= \inf  \{ \gamma_{AM}'(\widetilde{\phi}) ~|~ \widetilde{\phi}_1 = \phi \}.$$
As we mentioned in the introduction, the missing ingredient to make $\gamma_{AM}$ a conjugation-invariant norm is the triangle inequality, i.e. an analogue of 3) from Proposition \ref{prop: C_integer}. 

The only property we will need to prove Theorem \ref{Thm: AM_Continuity} is the fact that $c_{AM}(\widetilde{\phi})\in Spec_\alpha (\phi).$ As stated, this is not proven in \cite{AM18}, since there is an ambiguity in the definition of the spectrum we used and the definition of the spectrum in \cite{AM18}. In order to resolve this, we first recall the definition from \cite{AM18}.

Let $\phi \in \Contoc(W\times S^1)$ and $\{ \phi_t \}, t\in [0,1]$ such that $\phi_0=id,\phi_1=\phi.$ Let $z\in W\times S^1$ be a translated point of $\phi$ and $\tau \in \R$ such that $\Phi^\alpha_{-\tau} (\phi(z))=z.$ Note that, since the Reeb flow $\Phi^\alpha$ is 1-periodic, the set of all such $\tau$ is $\{ \mathcal{A}_\alpha(z)+\Z \}.$ Let $P$ be a torus with a disc removed and note that $W\times S^1=W\times \partial P\subset W\times P.$ We say that the pair $(z,\tau)$ is \textit{contractible} if the loop
$$l_z(t)= \begin{cases}
      \phi_{2t}(z), & \text{for}\ 0\leq t \leq 1/2 \\
      \Phi^{\alpha}_{(1-2t)\tau}(\phi(z)), & \text{for}\ 1/2 \leq t \leq 1 \\
      \end{cases}$$
is contractible in $W\times P.$ The spectrum considered in \cite{AM18} is given by
$$Spec_{\alpha}'(\phi)=\{ \tau ~|~ (z,\tau) \text{ is contractible} \}.$$
This definition comes from the fact that contractible $(z,\tau)$ are critical points of the Rabinowitz action functional.  This functional is defined on the space of contractible loops in $W\times \overline{P}$. where $\overline{P}$ denotes the completion of $(P,d\lambda_P), \lambda_P|_{\partial P}=d\theta.$ From this viewpoint, $W\times \overline{P}$ is seen as symplectization of $W\times S^1$ with the negative end compactified by $P.$

\begin{lemma}\label{Lemma:Actions_vs_Actions}
For all $\phi \in \Contoc(W\times S^1)$,  $Spec'_\alpha(\phi)\subset Spec_\alpha(\phi).$ In particular, for all $\widetilde{\phi} \in \ContocU(W\times S^1)$,  $c_{AM}(\widetilde{\phi}) \in Spec_\alpha(\phi).$
\end{lemma}
\begin{proof}
Let $(z,\tau)\in Spec'_\alpha(\phi).$ We claim that $l_z(t)$ is contractible in $W\times S^1$ as well.  This follows from the fact that the inclusion-induced map $\widetilde{\pi}_1(W\times S^1) \to \widetilde{\pi}_1(W\times P)$ is an injection, $\widetilde{\pi}_1$ denoting the set of homotopy classes of free loops. This fact can be proven by identifying $\widetilde{\pi}_1$ with conjugacy classes in $\pi_1$ and using the fact that $P$ deformation-retracts to $S^1 \vee S^1.$ Now
$$0=\int_{S^1} l_z^*(d\theta)=\mathcal{A}_\phi(z)-\tau,$$
and the claim follows.
\end{proof}


\section{Proof of Theorems \ref{thm: Rokhlin property of the contactomorphisms of the ball},  \ref{thm:Rok_Preq} and \ref{thm:Roklin_Stabilization}}\label{Section:Rokhlin}

\subsection{Squeezing implies Rokhlin property}

Here we prove the first part of Theorem \ref{thm: Rokhlin property of the contactomorphisms of the ball}, claiming that $(\ContocH(\R^{2n+1}), \tau_{C^0})$ has Rokhlin property. 

\begin{proof}

We will construct an explicit $g\in \ContocH(\R^{2n+1})$ such that $\conjH(g)=\ContocH(\R^{2n+1}).$ Throughout the proof, we repeatedly use the following relation which holds for all $\phi,\psi \in \ContocH(\R^{2n+1}):$
\begin{equation}\label{eq:Supports}
\supp (\psi \phi \psi^{-1}) =\psi(\supp(\phi)).
\end{equation}

By Proposition \ref{Prop:separable},  $(\Contoc(\R^{2n+1}), \tau_{C^0})$ is separable and we choose a sequence $\{f_i \in \Contoc(\R^{2n+1}) \}_{i\geq 1}, $ which is dense in $(\Contoc(\R^{2n+1}), \tau_{C^0})$. Note that $\{f_i\}_{i \geq 1}$ is dense in $(\ContocH(\R^{2n+1}), \tau_{C^0})$ as well.  The Rokhlin element $g$ constitutes of disjoint copies of $\{ f_i \}_{i\geq 1}$ shrunk into smaller and smaller balls, see Figure \ref{Figure2}. We divide the proof into two parts - the construction of $g$ and the proof of the Rokhlin property.

\textbf{Construction of $g$:} For every $i\geq 1$ let $\psi_i = \Psi_{\frac{1}{2^i}}$ be the squeezing contactomorphism given by Lemma \ref{Lemma:Ball_Sqeezing}, such that $\psi_i(\supp(f_i))\subset B(\frac{1}{2^i}).$  By (\ref{eq:Supports}), we have that
\begin{equation}\label{eq:support_small}
\supp(\psi_i f_i \psi_i^{-1}) \subset B\left( \frac{1}{2^i} \right).
\end{equation}

For $t\in \R$, let $T_t:\R^{2n+1}\to \R^{2n+1}$ be a translation by $t$ in the $x_1$-direction, i.e.
$$T_t(x_1,\ldots, x_n,y_1,\ldots,y_n,z)=(x_1+t,\ldots, x_n,y_1,\ldots,y_n,z).$$
Since $T_t^*\alpha_0 = \alpha_0$, each $T_t$ is a contactomorphism, which is moreover connected to identity as $T_0=id.$ Let $H_t$ be the contact Hamiltonian which generates $T_t.$ Taking a cut-off function $\rho$, which is compactly supported and equal to 1 on $B(1000)$, we get that $\rho H_t$ generates a family $T'_t\in \Contoc(\R^{2n+1})$,  such that for $|t|\leq 100$, $T'_t=T_t$ on $B(100).$ Let $\varphi_i = T'_{2-\frac{3}{2^i}} \circ \psi_i.$ The translation factor $2-\frac{3}{2^i}$ is chosen so that $\varphi_i f_i \varphi_i^{-1}$ is supported between hyperplanes $\{x_1=2-\frac{1}{2^{i-2}} \}$ and $\{x_1=2-\frac{1}{2^{i-1}} \}.$ More precisely, by (\ref{eq:Supports})
$$\supp(\varphi_i f_i \varphi_i^{-1}) =\supp (T'_{2-\frac{3}{2^i}} \psi_i f_i \psi_i^{-1} (T'_{2-\frac{3}{2^i}})^{-1}) = T'_{2-\frac{3}{2^i}} (\supp( \psi_i f_i \psi_i^{-1} )),$$
and thus, by (\ref{eq:support_small}), $\varphi_i f_i \varphi_i^{-1}$ is supported inside a ball of radius $\frac{1}{2^i}$ centered at $2-\frac{3}{2^i},$ see Figure \ref{Figure1}.

\begin{figure}[ht]
	\begin{center}
		\includegraphics[scale=0.63]{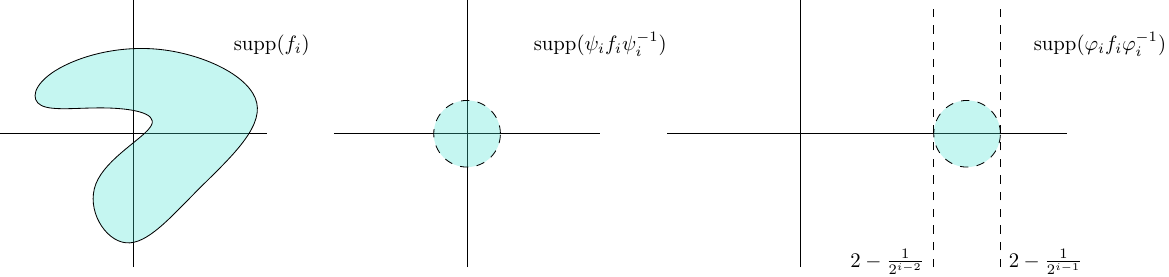}
		\caption{Supports of different conjugates}
		\label{Figure1}
	\end{center}
\end{figure}

Finally, we define $g$ as
$$g=\varphi_1 f_1 \varphi_1^{-1} \circ \varphi_2 f_2 \varphi_2^{-1} \circ \varphi_3 f_3 \varphi_3^{-1} \circ \ldots$$
Note that the supports of $\varphi_i f_i \varphi_i^{-1}$ are all disjoint and contained in $B(2).$ Thus $\supp(g)\subset B(2)$ and $g$ is a well-defined element of $\ContocH(\R^{2n+1})$ since 
$$\diam (\supp (\varphi_i f_i \varphi_i^{-1}))\leq \frac{1}{2^{i-1}}.$$
It is shown on Figure \ref{Figure2}.

\begin{figure}[ht]
	\begin{center}
		\includegraphics[scale=0.63]{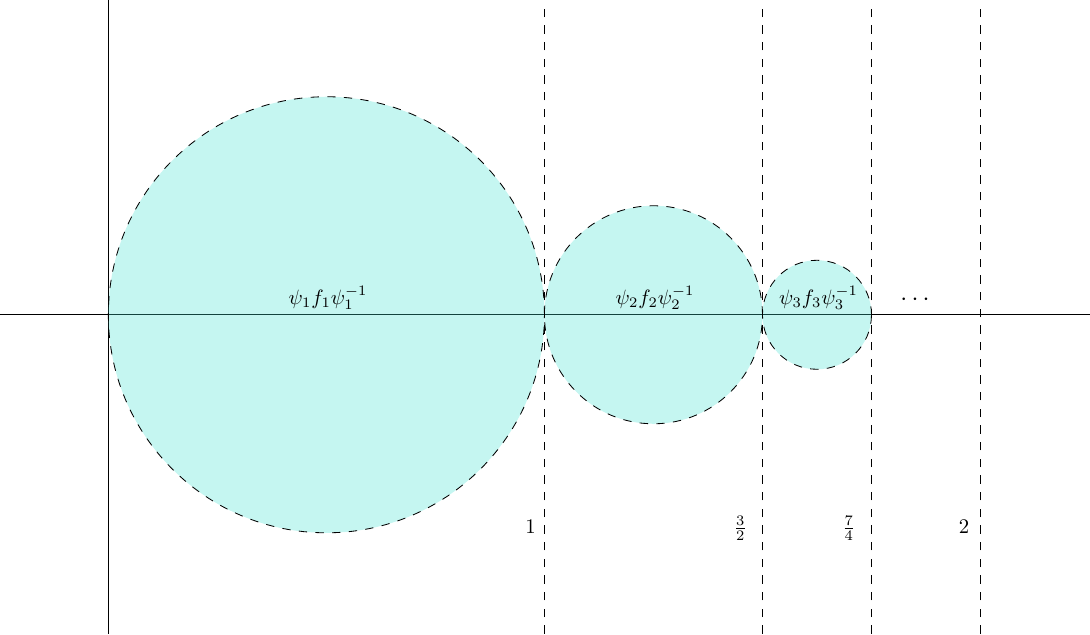}
		\caption{Rokhlin element $g$}
		\label{Figure2}
	\end{center}
\end{figure}

\textbf{The proof of Rokhlin property:} We claim that $\conjH(g)=\ContocH(\R^{2n+1}).$ By the choice of $\{f_i \}$, it is enough to prove that $f_i \in \conjH(g)$ for all $i\geq 1.$ Let us fix $i.$ Since $\varphi_i f_i \varphi_i^{-1}$ is compactly supported in the open strip $2-\frac{1}{2^{i-2}}<x_i <2-\frac{1}{2^{i-1}}$, there exists an $\varepsilon>0$ such that
$$ \supp (\varphi_i f_i \varphi_i^{-1})\subset \left\lbrace 2 -\frac{1}{2^{i-2}}+\varepsilon<x_i <2-\frac{1}{2^{i-1}}-\varepsilon \right\rbrace.$$
We wish to move the $\varphi_i f_i \varphi_i^{-1}$-part of $g$ far away from the rest of the support of $g.$ To this end, note that the Reeb flow $\Phi^{\alpha_0}_t$ is given by
$$\Phi^{\alpha_0}_t(x_1,\ldots, x_n,y_1, \ldots,y_n,z)=(x_1,\ldots, x_n,y_1, \ldots,y_n,z+t).$$
It is generated by a constant contact Hamiltonian equal to 1. We will cut it off inside a large box, which is contained in $\{2-\frac{1}{2^{i-2}}+\varepsilon<x_i <2-\frac{1}{2^{i-1}}-\varepsilon\}$ and which contains $\supp(\varphi_i f_i \varphi_i^{-1}).$ More precisely, let $\rho$ be a compactly supported function which is equal to 1 on the box $[2-\frac{1}{2^{i-2}}+\varepsilon,2-\frac{1}{2^{i-1}}-\varepsilon]\times [-1000,1000]^{2n}$ and equal to 0 whenever $x\leq 2-\frac{1}{2^{i-2}}$ or $x\geq 2-\frac{1}{2^{i-1}}.$ Denote by $\Phi'_t$ the isotopy generated by $\rho.$ For $|t|\leq 100$ it acts as $\Phi^{\alpha_0}_t$ on $\supp (\varphi_i f_i \varphi_i^{-1}).$ We let $\Phi'=\Phi'_{100}$ be the time-100 map of this isotopy. Since $\Phi'$ acts as $id$ on $\supp(\varphi_j f_j \varphi_j^{-1})$ for all $j\neq i$, we have that
$$\Phi' g (\Phi')^{-1}=\varphi_1 f_1 \varphi_1^{-1} \circ \ldots \circ \varphi_{i-1} f_{i-1} \varphi_{i-1}^{-1} \circ \Phi' \varphi_i f_i \varphi_i^{-1} (\Phi')^{-1} \circ \varphi_{i+1} f_{i+1} \varphi_{i+1}^{-1} \circ \ldots$$
The support of $\Phi' g (\Phi')^{-1}$ is shown on Figure \ref{Figure3}.

\begin{figure}[ht]
	\begin{center}
		\includegraphics[scale=0.63]{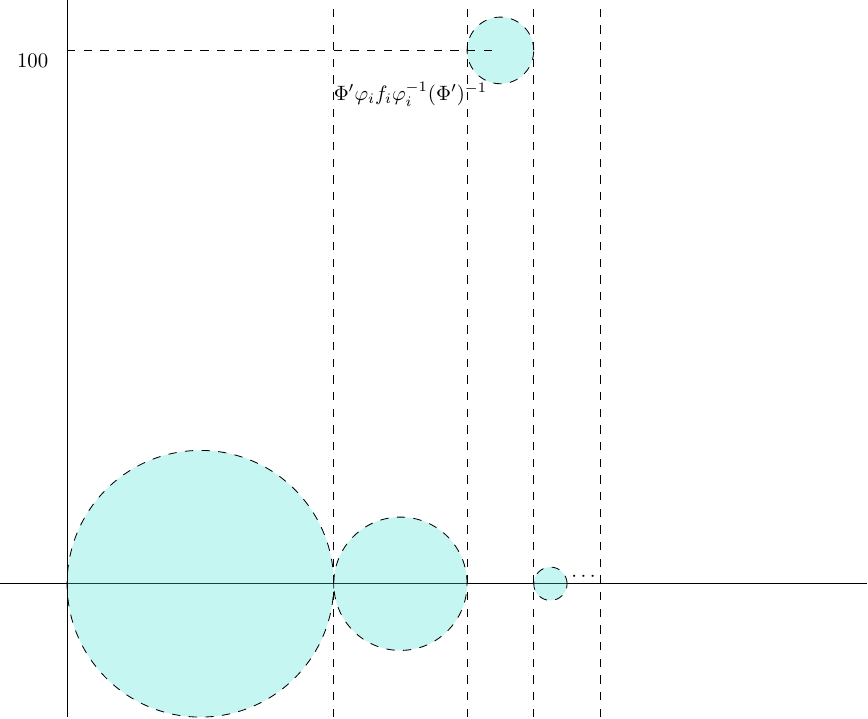}
		\caption{The support of $\Phi' g (\Phi')^{-1}$}
		\label{Figure3}
	\end{center}
\end{figure}

As a last step, we shrink all $\supp(\varphi_j f_j \varphi_j^{-1})$ for $j\neq i.$ More precisely, let $\Psi_\epsilon$ be the contactomorphism given by Lemma \ref{Lemma:Ball_Sqeezing}, supported inside $B(4)$ and such that $\Psi_\epsilon(B(3))\subset B(\epsilon).$ Since $\Psi_\epsilon$ and $\Phi' \varphi_i f_i \varphi_i^{-1} (\Phi')^{-1}$ have disjoint supports, the support of $\Psi_\epsilon \Phi' g (\Phi')^{-1} \Psi_\epsilon^{-1}$ splits into two parts - $\supp(\Phi' \varphi_i f_i \varphi_i^{-1} (\Phi')^{-1})$ and the rest, which is contained inside $B(\epsilon)$, see Figure \ref{Figure4}.

\begin{figure}[ht]
	\begin{center}
		\includegraphics[scale=0.63]{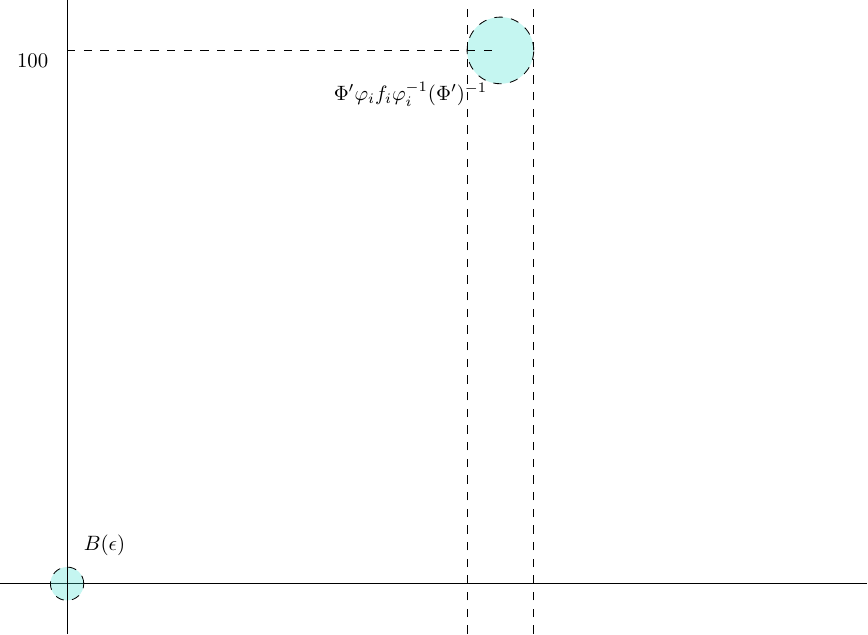}
		\caption{The support of $\Psi_\epsilon \Phi' g (\Phi')^{-1} \Psi_\epsilon^{-1}$}
		\label{Figure4}
	\end{center}
\end{figure}

Now, it readily follows that
$$d_{C^0}( \Phi' \varphi_i f_i \varphi_i^{-1} (\Phi')^{-1}, \Psi_\epsilon \Phi' g (\Phi')^{-1} \Psi_\epsilon^{-1})<2\epsilon.$$
Letting $\epsilon \to 0$ we conclude that $\Phi' \varphi_i f_i \varphi_i^{-1} (\Phi')^{-1} \in \conjH(g). $ Since $\conjH(g)$ is conjugation-invariant, $f_i\in \conjH(g)$, which finishes the proof.
\end{proof}

\subsection{Non-squeezing implies no Rokhlin property}

First we prove the second part of Theorem \ref{thm: Rokhlin property of the contactomorphisms of the ball} claiming that $(\ContocH(\R^{2n}\times S^1), \tau_{C^0})$ does not have Rokhlin property and then we prove Theorems \ref{thm:Rok_Preq} and \ref{thm:Roklin_Stabilization}.

\begin{proof}[Proof of the second part of Theorem \ref{thm: Rokhlin property of the contactomorphisms of the ball}]
First we observe that contact non-squeezing holds for elements of $\ContocH(\R^{2n}\times S^1).$ Indeed, assume that $1\leq \pi r^2< \pi R^2$ and $f\in \ContocH(\R^{2n}\times S^1)$ such that $f(B(R)\times S^1)\subset B(r)\times S^1.$ Then, for every $\varepsilon>0$ there exists $f_\varepsilon\in \Contoc(\R^{2n}\times S^1)$ such that $f_\varepsilon(B(R)\times S^1)\subset B(r+\varepsilon)\times S^1.$ Taking $\varepsilon<R-r$ violates contact non-squeezing in $\Contoc(\R^{2n}\times S^1).$

Now, we argue by contradiction. Assume that $g\in \ContocH(\R^{2n}\times S^1)$ is such that $\conjH(g)= \ContocH(\R^{2n}\times S^1)$ and $\supp(g)\subset B(r)\times S^1$.  Let $f \in \ContocH(\R^{2n}\times S^1)$ be such that $f(x)\neq x$ for all $x\in  \overline{B}(r+1)\times S^1.$ To construct such an $f$ it is enough to take a compactly supported Hamiltonian equal to $\frac{1}{2}$ in an open neighbourhood of $\overline{B}(r+1)\times S^1.$ This Hamiltonian generates a half-rotation on $\overline{B}(r+1)\times S^1.$ We claim that $f\notin \conjH(g).$ Indeed, assume there exists a sequence $h_n \in \ContocH(\R^{2n}\times S^1)$ such that $h_n g h_n^{-1}\to f.$ We observe that
$$B(r+1)\times S^1 \not\subset \supp (h_n g h_n^{-1}) = h_n(\supp(g)), $$
since otherwise 
$$h_n^{-1}( B(r+1)\times S^1 ) \subset \supp(g) \subset B(r)\times S^1, $$
contradicting contact non-squeezing - Theorem \ref{Thm:Non-squeezing_Original}. Thus, there exist $x_n\in B(r+1)\times S^1 $ such that $h_n g h_n^{-1}(x_n)=x_n.$ Passing to a subsequence, we may assume that $x_n\to x \in \overline{B}(r+1)\times S^1.$ We obtain that
$$f(x)=\lim_{n\to \infty} h_n (g (h_n^{-1}(x_n))) = \lim_{n\to \infty}  x_n =x,$$
which contradicts the choice of $f.$
\end{proof}

\begin{proof}[Proof of Theorem \ref{thm:Rok_Preq}]
Let $W_0$ be the Liouville domain whose completion is $W$ and $\mathring{W}_0$ its interior. Since $SH(W)=0$, by \cite[Theorem 1.3]{BK22} we have that $0<c_{sp}(\mathring{W}_0)< +\infty$ and we may consider $U=\frac{1}{c_{sp}(\mathring{W}_0)} \mathring{W}_0$ whose symplectic capacity is 1. Rescalings are diffeomorphisms and thus $rU$ has compact closure for any $r>0.$ Now, we repeat the argument from the above proof, substituting $B(r)\times S^1$ with $(rU)\times S^1$ and using Theorem \ref{Thm:Non-squeezing_AM} instead of Theorem \ref{Thm:Non-squeezing_Original}.
\end{proof}

\begin{proof}[Proof of Theorem \ref{thm:Roklin_Stabilization}]
Assume that the conjugacy class of $g\in \ContocH(W\times \R^{2m} \times S^1)$ is dense and let $r_0,R>0$ be such that $\supp(g) \subset (R \mathring{W}_0) \times B(r_0) \times S^1.$ Since $c_{HZ}(R \mathring{W_0})=R \cdot c_{HZ}( \mathring{W_0})<+\infty,$ we may apply Theorem \ref{Thm:Non-squeezing_Stabilization} by taking $R$ sufficiently large so that $\lceil \pi r_0^2 \rceil < \lceil c_{HZ}(R\mathring{W}_0) \rceil.$ We conclude that for $r_1\geq \sqrt{\frac{c_{HZ}(R\mathring{W}_0)}{\pi}} +1$, $(R \mathring{W}_0) \times B(r_1) \times S^1$ cannot be squeezed into $(R \mathring{W}_0) \times B(r_0) \times S^1.$ Taking $f$ which moves every point of $(R W_0) \times \overline{B}(r_1) \times S^1$ shows that $f\notin \conjH(g)$ by the same argument as in the two proofs above.
\end{proof}


\section{$C^0$-continuity of the contact action spectrum}

We wish to prove Theorems \ref{Gamma_Is_Locally_Bounded} and  \ref{Thm: AM_Continuity}. To this end, we will show a general $C^0$-continuity result about the contact action spectrum.

\subsection{Action spectrum continuity} Let $(W\,d\lambda)$ be an exact symplectic manifold and $W\times S^1$ its prequantization space. Let $\pi: W\times S^1 \to S^1$ be the projection and assume $S^1=\R / \Z$ is equipped with the standard flat metric $ g_{std}$ induced from $\R.$

\begin{prop}\label{Thm:C0-continuity}
 If $\phi \in \Contoc(W\times S^1)$ is such that $d_{C^0} (\pi,  \pi \circ \phi) < \frac{1}{2}$, then it holds
$$Spec_\alpha(\phi)\subset\left[-d_{C^0} (\pi,  \pi \circ \phi) , d_{C^0} (\pi,  \pi \circ \phi)\right].$$
\end{prop}

\begin{proof}
Let $\{ \phi_t \}_{t\in [0,1]}$ be a path in $\Contoc(W \times S^1)$ such that $\phi_0=id, \phi_1=\phi.$ Define $F:W\times S^1 \to \R$ by $F(x)=\int_0^1 d\theta(\dot{\phi}_t(x)) dt.$ By definition, for every translated point $z$ of $\phi$,  $F(z)=\mathcal{A}_\alpha(z).$

Firstly, we show that all translated points have actions in the interval $(-1/2, 1/2).$ To this end, assume that there exists a translated point $z$ whose action is not in this window.  We have that $|F(z)|\geq 1/2.$ On the other hand, for any $x$ outside of the joint support of all $\phi_t$, we have that $\phi_t(x)=x$ for all $t$, and thus $F(x)=0.$ Since $W\times S^1$ is connected, there exists $y\in W\times S^1$ such that $|F(y)|=1/2.$ By the definition of $F$, $\pi(y)$ and $\pi(\phi(y))$ are antipodal points on $S^1$ which implies that $d_{C^0}(\pi, \pi \circ \phi) \geq d_{g_{std}}(\pi(y), \pi(\phi(y)))=1/2.$ This contradicts the assumption of the theorem.

We have thus proven that for every translated point $z$, $\mathcal{A}_\alpha(z)= F(z)\in (-1/2,1/2).$ However, since the length of $S^1$ is equal to 1, we have that 
$$\mathcal{A}_\alpha(z)= F(z)=d_{g_{std}}(\pi(z), \pi(\phi(z)))\leq d_{C^0}(\pi, \pi \circ \phi),$$
which finishes the proof. 
\end{proof}

The expression $d_{C^0} (\pi,  \pi \circ \phi)$ can be considered to be the $C^0$-norm of $\phi$ with respect to the product pseudometric $0 \oplus g_{std}$, where $0$ denotes the zero pseudometric on $W.$ Since the zero pseudometric is the smallest possible, the same conclusion holds for any other metric on $W$ and we obtain the following corollary.

\begin{coro}\label{Coro:C0-continuity}
Let $W\times S^1$ be equipped with a product metric $g \oplus g_{std}$, where $g_{std}$ is induced by the standard flat metric on $\R$ and $g$ is an arbitrary Riemannian metric on $W.$ For all $\phi \in \Contoc(W\times S^1)$ such that $\| \phi \|_{C^0}< \frac{1}{2}$, it holds
$$Spec_\alpha(\phi)\subset\left[-\| \phi \|_{C^0} , \| \phi \|_{C^0}\right].$$
\end{coro}
\begin{proof}
For any Riemannian metric $g$, $d_{C^0} (\pi,  \pi \circ \phi)\leq  \| \phi \|_{C^0} $ and the claim follows from Proposition \ref{Thm:C0-continuity}.
\end{proof}

\subsection{Proof of Theorems \ref{Gamma_Is_Locally_Bounded} and \ref{Thm: AM_Continuity}}\label{Section:Gamma_Proof}

\begin{proof}[Proof of Theorem \ref{Gamma_Is_Locally_Bounded}]
As a direct consequence of Corollary \ref{Coro:C0-continuity} and Proposition \ref{prop:C_non_integer} we obtain that for $\phi \in \Contoc(\R^{2n}\times S^1)$ such that $\| \phi \|_{C^0}< \frac{1}{2}$ Sandon's spectral invariants satisfy
\begin{equation}\label{eq: Spectral_Continuity}
-\| \phi \|_{C^0} \leq c_{-}(\phi) \leq 0 \leq c_{+}(\phi) \leq \| \phi \|_{C^0}.
\end{equation}

Thus, for such $\phi$, $\gamma(\phi)\leq 2.$ Applying this inequality to $\phi \psi^{-1}$ yields $|\gamma(\phi)-\gamma(\psi)| \leq \gamma(\phi \psi^{-1})\leq 2$ as needed.
\end{proof}

\begin{rem}\label{Remark:c+_continuity}
(\ref{eq: Spectral_Continuity}) also implies that if $\| \phi \|_{C^0} < \frac{1}{2}$ then $\lceil c^+(\phi) \rceil \leq 1$, which further implies that if $d_{C^0}(\phi,\psi)<\frac{1}{2}$ then
$$| \lceil c^+(\phi) \rceil - \lceil c^+(\psi) \rceil | \leq \max (\lceil c^+(\phi\psi^{-1}) \rceil, \lceil c^+(\psi \phi^{-1}) \rceil) \leq 1.$$
\end{rem}

\begin{proof}[Proof of Theorem  \ref{Thm: AM_Continuity}]
The theorem is a direct consequence of Lemma \ref{Lemma:Actions_vs_Actions} and Corollary \ref{Coro:C0-continuity}.
\end{proof}


\section{$C^0$-properties of the spectral metric}

\subsection{Proof of Proposition \ref{Lower_Semi_Gamma}}

We will prove that 
$$\lceil c^+(\cdot) \rceil,-\lfloor c^-(\cdot) \rfloor: (\Contoc(\R^{2n}\times S^1),\tau_{C^0}) \to \Z$$
 are lower semicontinuous. From here Proposition \ref{Lower_Semi_Gamma} immediately follows since $\gamma(\cdot) = \lceil c^+(\cdot) \rceil-\lfloor c^-(\cdot) \rfloor.$
 
We start by proving lower semicontinuity of $\lceil c^+(\cdot) \rceil.$ Let $\phi \in \Contoc(\R^{2n}\times S^1)$ and denote by $\varepsilon=\lceil c^+(\phi) \rceil - c^+(\phi)$, $\varepsilon \in [0,1).$ Assume that $\psi \in \Contoc(\R^{2n}\times S^1)$ is such that $d_{C^0}(\phi, \psi) = \| \phi \psi^{-1} \|_{C^0} < \min (1-\varepsilon, 1/2).$ Using (\ref{eq: Spectral_Continuity}), we estimate
$$\lceil c^+(\phi) \rceil - 1 = \lceil c^+(\phi) \rceil - \varepsilon - (1-\varepsilon)< c^+(\phi) -  \| \phi \psi^{-1} \|_{C^0}  \leq c^+(\phi) -  c^+(\phi \psi^{-1} ) \leq \lceil c^+(\phi) \rceil,$$
which implies that $\lceil c^+(\phi) \rceil = \lceil c^+(\phi) -  c^+(\phi \psi^{-1} )  \rceil .$ Using 3) from Proposition \ref{prop: C_integer} we get that
\begin{equation}\label{LSC_C+}
\lceil c^+(\phi) \rceil = \lceil c^+(\phi) -  c^+(\phi \psi^{-1} )  \rceil \leq \lceil c^+(\psi) \rceil, 
\end{equation}
which proves the claim.

To prove lower semicontinuity of $-\lfloor c^-(\cdot) \rfloor$ we fix again $\phi \in \Contoc(\R^{2n}\times S^1)$ and use (\ref{LSC_C+}) for $\phi^{-1}$ insted of $\phi.$ This implies that there exists a $C^0$-neighbourhood $U$ of $\phi^{-1}$ such that for all $\psi \in U$ it holds
$$\lceil c^+(\phi^{-1}) \rceil \leq \lceil c^+(\psi) \rceil.$$
By 2) from Proposition \ref{prop: C_integer}, we have that for all $\psi \in U$
\begin{equation}\label{LSC_C+2}
-\lfloor c^-(\phi) \rfloor \leq -\lfloor c^-(\psi^{-1}) \rfloor.
\end{equation}
 Equivalently, (\ref{LSC_C+2}) holds whenever $\psi ^{-1} \in U^{-1} = \{ \varphi^{-1} ~|~ \varphi \in U \}.$ By Lemma \ref{Lemma:Top_Group}, taking inverses is a continuous map an hence $U^{-1}$ is an open neighbourhood of $\phi,$ which finishes the proof. \qed

\subsection{Proof of Theorem \ref{Gamma_Homeo}}


\begin{proof}
First, we show that $\widetilde{\gamma}$ coincides with $\gamma$ on  $\Contoc(\R^{2n}\times S^1).$ For $\phi \in \Contoc(\R^{2n}\times S^1)$, $\widetilde{\gamma}(\phi)\leq \gamma(\phi)$ by definition of $\widetilde{\gamma}$ since $\phi \in U'$ for every open neighbourhood $U.$ On the other hand, by Proposition \ref{Lower_Semi_Gamma}, there exists an open set $U \ni \phi $ in $ \ContocH(\R^{2n}\times S^1)$ such that $\gamma (\phi)\leq \gamma( \psi)$ for all $\psi \in U'.$ Thus, $\gamma(\phi) \leq \min_{\psi \in U'} \gamma(\psi) \leq \widetilde{\gamma}(\phi).$

To show non-degeneracy assume $\phi \in \ContocH(\R^{2n}\times S^1)$, $\phi \neq id.$ Then there exists an open neighbourhood $U$ of $\phi$ such that $id \notin U$ and thus $\min_{\psi \in U'} \gamma(\psi)\geq 1.$ This implies $\gamma(\phi)\geq 1.$

Next, we show that $\widetilde{\gamma}$ is symmetric.  Since taking inverses is a homeomorphism of $(\ContocH(\R^{2n}\times S^1),\tau_{C^0})$, see Lemma \ref{Lemma:Top_Group}, $U$ is an open neighbourhood of $\phi$ if and only if $U^{-1}=\{ \psi^{-1} ~|~ \psi \in U \}$ is an open neighbourhood of $\phi^{-1}.$ Using symmetry of $\gamma$ we have that
$$\widetilde{\gamma}(\phi) = \max_U \min_{\psi \in U'} \gamma(\psi) = \max_{U^{-1}} \min_{\psi^{-1} \in (U')^{-1}}  \gamma(\psi^{-1}) = \max_{U^{-1}} \min_{\psi^{-1} \in (U^{-1})'}  \gamma(\psi^{-1})=\widetilde{\gamma}(\phi^{-1}).  $$

Finally, we prove triangle inequality. Let $\phi, \psi \in \ContocH(\R^{2n}\times S^1).$ Since $\widetilde{\gamma}$ is $\Z$ valued, by definition there exists an open neighbourhood $U$ of $\phi \psi$ such that $\widetilde{\gamma}(\phi \psi)=\min_{\varphi \in U'} \gamma(\varphi).$ Similarly, there exist sequences $\phi_k,\psi_k \in \Contoc(\R^{2n}\times S^1)$, $k\geq 1$ such that $\phi_k \to \phi, \psi_k \to \psi$ in $\tau_{C^0}$ and $\gamma(\phi_k)=\widetilde{\gamma}(\phi), \gamma(\psi_k)=\widetilde{\gamma}(\psi)$ for all $k.$ Again using Lemma \ref{Lemma:Top_Group}, we have that composition is continuous, and thus $\phi_k \psi_k \in U'$ for sufficiently large $k.$ Using triangle inequality for $\gamma$ we get that
$$\widetilde{\gamma}(\phi\psi)=\min_{\varphi \in U'} \gamma (\varphi) \leq \gamma (\phi_k \psi_k)\leq \gamma (\phi_k)+ \gamma (\psi_k) = \widetilde{\gamma}(\phi) + \widetilde{\gamma}(\psi).$$

Let us now show that $\widetilde{\gamma}$ is locally bounded. Let $U$ be an open neighbourhood of $id$ in $\ContocH(\R^{2n}\times S^1)$ such that $\gamma <C$ on $U'.$ Such $U$ exists since $\gamma$ is locally bounded by Theorem \ref{Gamma_Is_Locally_Bounded} (in fact we can take $U$ to be the open ball of radius $\frac{1}{2}$ in $d_{C^0}$ and $C=2$). Let $\phi \in U$ and let $V$ be a sufficiently small open neighbourhood of $\phi$ so that $V\subset U$ and $\widetilde{\gamma}(\phi)=\min_{\psi \in V'} \gamma(\psi).$ Now $\widetilde{\gamma}(\phi)=\min_{\psi \in V'} \gamma(\psi)\leq \min_{\psi \in U'} \gamma(\psi) <C.$

Lastly, we show that $\widetilde{\gamma}$ is lower semicontinuous. Assume the contrary, that there exists $\phi \in \ContocH(\R^{2n}\times S^1)$ such that for every open neighbourhood $U \ni \phi$ there exists $\psi \in U$, $\widetilde{\gamma}(\psi)< \widetilde{\gamma}(\phi).$ Be the definition of $\widetilde{\gamma}$ there exists an open neighbourhood $V\subset U$ of $\psi$ such that $\min_{\varphi \in V'} \gamma(\varphi)=\widetilde{\gamma}(\psi).$ Thus, in every neighbourhood $U$ of $\phi$ there exists $\varphi \in U'$ such that $\gamma(\varphi)<\widetilde{\gamma}(\phi)$, which is a contradiction. 
\end{proof}


\section{Proof of Theorem \ref{Thm:Conjugation_Comparison}}

In order to relate $\| \cdot \|_{conj}$ to $\| \cdot \|_{frac}$ we will need the following auxiliary statement.

\begin{lemma}\label{lem: closure of the conjugacy class in the contact group} Let $\phi \in \Contoc (Y,\xi)$ be supported inside a Darboux ball. There exist $\psi_i\in  \Contoc (Y,\xi), i \geq 1$ such that $\psi_i \phi \psi_i^{-1}$ converges to $id$ in the strong topology.
\end{lemma}
\begin{proof}
Working in local Darboux coordinates,  let $B(R)\subset \R^{2n+1}$ be the Darboux ball inside of which $\phi$ is supported. Due to compactness of $\supp (\phi)$, there exists $0<r<R$ such that $\supp (\phi) \subset B(r).$ Let $\psi_i=\Psi_{1/i}$, where $\Psi_{1/i}$ is the map given by Lemma \ref{Lemma:Ball_Sqeezing}, shrinking $B(r)$ and supported inside $B(\frac{R+r}{2}).$ Since $\supp(\psi_i \phi \psi_i^{-1})=\psi_i (\supp(\phi))$, we have that $\psi_i \phi \psi_i^{-1}\to id$ in $\tau_{C^0}$ defined by the standard flat metric on $\R^{2n+1}.$ Since all the maps are supported in $B(\frac{R+r}{2})$, Lemma \ref{Lemma:Strong_Convergence} implies that $\psi_i \phi \psi_i^{-1}\to id$ in the strong topology as well.
\end{proof}

\begin{rem}
Lemma $\ref{lem: closure of the conjugacy class in the contact group}$ implies that $\|\cdot \|_{conj}$ is trivial on $\Contoc(\R^{2n+1})$. 
\end{rem}

\begin{proof}[Proof of Theorem \ref{Thm:Conjugation_Comparison}]
Let
\begin{equation}\label{eq:Fragmentation_proof}
\phi=\phi_1 \circ \ldots \circ \phi_N,
\end{equation}
be a decomposition of $\phi$ into fragments supported in Darboux balls. By Lemma \ref{lem: closure of the conjugacy class in the contact group}, for all $1\leq i \leq N$, $id \in \conjH(\phi_i)$ independently of the choice of a Riemannian metric on $Y.$ This proves (\ref{Inequality_Conj_Frac}). 

To prove (\ref{Inequality_Sandon_Conj_Frac}), assume decomposition (\ref{eq:Fragmentation_proof}), where now for all $1\leq i \leq N$, $id \in \conjH(\phi_i).$ Theorem \ref{Gamma_Is_Locally_Bounded} implies that $\gamma(\phi_i)\leq 2$ and triangle inequality for $\gamma$ gives us
$$\gamma(\phi)=\gamma( \phi_1 \circ \ldots \circ \phi_N)\leq \gamma(\phi_1)+\ldots +\gamma(\phi_N)\leq 2N,$$
which finishes the proof.
\end{proof}


\section{Proof of Theorem \ref{thm: Rokhlin property of the contactomorphisms of B2nS1} }\label{Sec:Proof_cc}
\begin{proof}
Assume that $f,g\in \ContocH(\R^{2n}\times S^1)$ are $k$-conjugation connected. Let $\phi_i,\psi_i$ be as in the definition of conjugation connectedness. Since for all $0\leq i \leq k-1$, $\phi_i,\phi_{i+1} \in \conjH(\psi_i)$, Theorem \ref{Gamma_Homeo} implies that there exists $C>0$ such that
$$|\widetilde{\gamma}(\psi_i) - \widetilde{\gamma}(\phi_i)| \leq C, ~|\widetilde{\gamma}(\psi_i)-\widetilde{\gamma}(\phi_{i+1})| \leq C,$$
which implies that $|\widetilde{\gamma}(\phi_{i+1})-\widetilde{\gamma}(\phi_i)|\leq 2C.$ Summing these inequalities over $i$ gives us that
$$|\widetilde{\gamma}(f)-\widetilde{\gamma}(g)|\leq 2Ck$$
and thus
$$|\widetilde{\gamma}(f)-\widetilde{\gamma}(g)|\leq 2Cd_{cc}(f,g).$$
Since $\widetilde{\gamma}$ extends $\gamma$ and $\gamma$ is unbounded, see Theorem \ref{thm:Gamma_Basic}, so is  $d_{cc}$.
\end{proof}

\begin{rem}\label{Remark:cc_c+}
In the above argument we can take $C=2$ and obtain 
$$|\widetilde{\gamma}(f)-\widetilde{\gamma}(g)|\leq 4d_{cc}(f,g).$$
The same argument, based on Remark \ref{Remark:c+_continuity}, shows that
$$|\lceil c^+(f) \rceil - \lceil c^+(g) \rceil |\leq 2 d_{cc}(f,g).$$
In fact, one can improve the constants 4 and 2 in the above inequalities to 2 and 1 respectively, by using lower semicontinuity of $\widetilde{\gamma}$ and $\lceil c^+(\cdot) \rceil$.
\end{rem}

\section{Auxiliary examples}\label{Sec:Examples}

In this section, we collect a number of examples mentioned throughout the text. Each one of them illustrates a certain point related to the results of the paper. The first one is complementary to Theorem \ref{Gamma_Is_Locally_Bounded} and it shows that $\gamma$ is not locally bounded in the compact open topology.
\begin{ex}\label{Example:CO_vs_C0}
Let $T_k :\R^{2n}\times S^1 \to \R^{2n} \times S^1, k\geq 1$ be a translation with respect to the first coordinate
$$T_k(x_1,y_1,\ldots , x_n,y_n,z)=(x_1+k,y_1, \ldots, x_n,y_n,z),$$
$f_0\in \Contoc(\R^{2n}\times S^1), f_0 \neq id$ arbitrary and define $f_k=T_k \circ f_0 \circ T_k^{-1}.$ Now, for all $k\geq 1$, $f_k \in \Contoc(\R^{2n}\times S^1)$ since $T_k^* \alpha_0 = \alpha_0$ and $\gamma(f_k)=\gamma (f_0)>0$ by definition. However, $\supp (f_k)=T_k(\supp(f_0))$ and thus $f_k \to id$ in $\tau_{CO}$ by Lemma \ref{Lemma:CO_Convergence}. This shows that $\gamma$ is not locally bounded with respect to $\tau_{CO}.$
\end{ex}

The next example shows that $c^-(\phi),c^+(\phi)$ may be arbitrarily close to zero, while $id \notin \conjH(\phi).$ In particular, we may have that $\gamma(\phi)\leq 2$, while $\|\phi \|_{conj}>1$, which implies that the first inequality in (\ref{Inequality_Sandon_Conj_Frac}) is not an equality.

\begin{ex}\label{Example: Not_Equal} 
Let $H:\R^{2n}\to \R$ be a compactly supported autonomous Hamiltonian and $\phi$ the Hamiltonian diffeomorphism generated by $H$. Taking $H$ to be sufficiently $C^2$-small, we have that Lemma \ref{Lemma:Lifts_Computations} applies and thus $c^-(\widehat{\phi})=\min H,c^+(\widehat{\phi})=\max H.$ It is clear that $\widehat{\phi}$ defined in this way may have $c^-(\phi),c^+(\phi)$ arbitrarily close to zero and we wish to arrange that $id \notin \conjH(\widehat{\phi})$ as well.  To this end, let $q^+\in \R^{2n}$ be such that $H(q^+)=\max H.$ By (\ref{Equation:Lift}), we have that
\begin{equation}\label{Equation_Rational_Rotation}
\widehat{\phi}(q^+,z)=(q^+, z+\max H \mod 1).
\end{equation}
Let $H$ be such that $H(q^+)=\frac{1}{m}$, where $m\geq 5$ is an integer. We claim that for every $\psi \in \Contoc(\R^{2n}\times S^1)$ it holds $\| \psi \widehat{\phi} \psi^{-1} \|_{C^0}\geq \frac{1}{m}$ and thus $id \notin \conjH(\widehat{\phi}).$ 

To prove this claim,  we assume the contrary, that $\| \psi \widehat{\phi} \psi^{-1} \|_{C^0}< \frac{1}{m}.$ Let $q_1=(q^+,0), q_2=(q^+,\frac{1}{m}), \ldots , q_m=(q^+,\frac{m-1}{m})$ be $m$ equally spaced points on $\{ q^+ \} \times S^1.$ By (\ref{Equation_Rational_Rotation}) we have that $\widehat{\phi}(q_i)=q_{i+1}, i=1,\ldots , m$, where $q_{m+1}=q_1.$ Denoting $S= \psi(\{ q^+ \} \times S^1)$ and $Q_i=\psi (q_i)$, we have that $Q_i$ are $m$ points on $S$ such that $\psi \widehat{\phi} \psi^{-1}(Q_i)=Q_{i+1}.$ By our assumption for $i=1,\ldots , m$, $dist(Q_i , Q_{i+1})<\frac{1}{m}<\frac{1}{2}$ and thus there exists a unique length-minimizing path $p_i$ connecting $Q_i$ to $Q_{i+1}.$  Denote by $p_i^{-1}$ the reversed path and by $p_m^{-1}  \ldots  p_1^{-1}$ the loop obtained by concatenation. It has length less than 1 and is thus contractible. Let $l_i:[\frac{i-1}{m},\frac{i}{m}]\to S, l_i(t)=\psi ( q^+ ,t)$ be the arcs $[Q_i Q_{i+1}]$ on $S.$ The loops $l_i  p_i^{-1}$ for different $i$ are all homotopic to each other. Indeed, since $\| \psi \widehat{\phi} \psi^{-1} \|_{C^0}< \frac{1}{m}$ we have that $dist(l_i(t),l_{i+1}(t))<\frac{1}{m}.$ On the other hand $dist(p_i^{-1}(t),p_{i+1}^{-1}(t))<\frac{2}{m}$ since the length of each $p_i$ is $<\frac{1}{m}.$ We conclude that for each $t$
$$dist( (l_i p_i^{-1})(t), (l_{i+1}  p_{i+1}^{-1})(t))<\frac{2}{m}<\frac{1}{2}.$$
Thus, the two loops are homotopic by a homotopy following a unique length minimizing path between $(l_i  p_i^{-1})(t)$ and $(l_{i+1}  p_{i+1}^{-1})(t)$ for each $t.$ Lastly, we decompose the homotopy class of $S$ as follows:
$$[S]=[S p_m^{-1} \ldots  p_1^{-1} ] = [l_1  p_1^{-1} ] [p_1  l_2  p_2^{-1}  p_1^{-1} ] \ldots [ p_1  p_2  \ldots  p_{m-1}  l_m  p_m^{-1}  \ldots  p_1^{-1}] = m [l_1  p_1^{-1}].$$
This is a contradiction since $\pi_1(\R^{2n}\times S^1)\cong \Z$ and $[S]$ is the generator.
\end{ex}

\begin{rem}\label{Remark:Simpler_Example}
One could show that $\frac{1}{2}\gamma \neq  \| \cdot \|_{conj}$ simply by lifting $\phi$ generated by an autonomous $C^2$-small non-negative Hamiltonian. Indeed, in this case $ c^-(\widehat{\phi})=0$ and $\gamma(\widehat{\phi}) = \lceil c^+(\widehat{\phi}) \rceil =1 $ by Lemma \ref{Lemma:Lifts_Computations}, while $\| \widehat{\phi} \|_{conj} \geq 1$ by definition.
\end{rem}

Lastly, we wish to provide an example showing that $d_{cc}$ is neither left nor right-invariant for $(\ContocH(\R^{2n}\times S^1),\tau_{C^0}).$ We will need the following lemma:

\begin{lemma}\label{lem: unboundedness of c}
For any real $a >0$, there exists an autonomous,  non-negative, compactly supported Hamiltonian $H: \R^{2n} \to \R$ such that for all $t \in [0,1]$,
\[c^+_{Ham}(\phi^H_t)=at.\]
\end{lemma}

\begin{proof}
In order to show the lemma, it suffices to construct $H$ such that $Spec(\phi^H_t)=\{0, at\}.$ Indeed, by Theorem \ref{Thm_Viterbo_Basic}, $c^-_{Ham}(\phi^H_t),c^+_{Ham}(\phi^H_t)\in Spec(\phi^H_t)$ and $c^-_{Ham}(\phi^H_t)\leq 0 \leq c^+_{Ham}(\phi^H_t).$ Since $\phi^H_t \neq id$ , $c^+(\phi^H_t)>0$. Let $F \colon \R_{\geq 0} \to \R_{\geq 0}$ be a compactly supported function such that $F(0)=a$ and for all $r> 0$
\[-2 \pi r < F'(r)< 0.\]

We now define $H(x)=F(\vert x \vert)$. For $x\neq 0$, we have that $\nabla H(x)=F'(|x|)\frac{x}{|x|}$ and thus the trajectory $\phi^H_t(x)$ follows a circle of radius $\vert x \vert$ centered at the origin at speed $F'(\vert x \vert)< 2 \pi \vert x \vert$. Hence, the only periodic orbits of $\phi^H_t$ are the critical points of $H$ and $Spec(\phi^H_t)=\{0, at \}.$
\end{proof}

\begin{ex}\label{Example_DCC}
Let us fix $a>0$ and let $H$ be a Hamiltonian given by Lemma \ref{lem: unboundedness of c}. Assume that $\psi \in \Ham_c(\R^{2n})$ displaces the support of $H.$ First, we claim that for all $0\leq t < T \leq 1$ it holds
\begin{equation}\label{eq:Displacement_Invariant}
c^+_{Ham}(\psi \phi^H_T \psi^{-1} \phi^H_t )=aT.
\end{equation} 
Indeed, consider the isotopy $\varphi_s = \psi \phi^H_T \psi^{-1} \phi^H_s$ , $s \in [0,t].$ Since $\psi \phi^H_T \psi^{-1}$ and $\phi^H_s$ have disjoint supports we have that
$$Spec(\psi \phi^H_T \psi^{-1} \phi^H_s)= Spec(\psi \phi^H_T \psi^{-1}) \cup Spec (\phi^H_s)= \{ 0,as,aT \}.$$
Since $\varphi_0 \neq id$ and $c^+_{Ham}(\varphi_0) \in \{0, aT \} $, we conclude that $c^+_{Ham}(\varphi_0)= aT.$ On the other hand, by continuity of $c^+_{Ham}$,  the map $s \to c^+_{Ham}(\varphi _s)\in \{ 0,as,aT \}$ is continuous and thus constant since $as<aT$ for all $s$. Taking $s=t$ proves the claim.

Now, applying (\ref{eq:Displacement_Invariant}) with $t=\frac{1}{3}, T=\frac{2}{3}$ we have that
$$ c^+_{Ham}(\psi \phi^H_{\frac{2}{3}} \psi^{-1} \phi^H_{\frac{1}{3}})=\frac{2}{3}a ,~ c^+_{Ham}(\phi^H_{\frac{2}{3}}  \phi^H_{\frac{1}{3}})= c^+_{Ham}(\phi^H_1)=a.$$
Let $f,g,h\in \Contoc(\R^{2n}\times S^1)$ be the lifts $f=\widehat{ \phi}^H_{\frac{2}{3}}$, $g=\widehat{ \phi}^H_{\frac{1}{3}}$ and $h=\widehat{\psi}.$ By definition, $d_{cc}(f,hfh^{-1})=1.$ On the other hand, using  (\ref{eq:Spec_Lifts}) and Remark \ref{Remark:cc_c+} we have that 
$$\frac{a}{3}-1 \leq \lceil c^+_{Ham} (\phi^H_{\frac{2}{3}}  \phi^H_{\frac{1}{3}}) \rceil - \lceil c^+_{Ham} (\psi \phi^H_{\frac{2}{3}} \psi^{-1} \phi^H_{\frac{1}{3}}) \rceil  =$$
$$= \lceil c^+(fg) \rceil - \lceil c^+(hfh^{-1}g) \rceil  \leq 2d_{cc}(fg,hfh^{-1}g).$$
Taking sufficiently large $a$ shows that $d_{cc}(f,hfh^{-1})\neq d_{cc}(fg,hfh^{-1}g)$, i.e. $d_{cc}$ is not right-invariant.  Since taking inverses preserves $d_{cc}$, $d_{cc}$ is also not left-invariant.
\end{ex}

\end{document}